\documentclass[12pt]{amsart}
\usepackage{amssymb}
\usepackage{amsmath}
\usepackage{eurosym}
\usepackage{graphics}
\usepackage{epsfig}
\usepackage{tabu}
\usepackage{graphicx}
\usepackage{cite}           
\usepackage{url}
\usepackage{epstopdf}
\usepackage{hyperref}
\usepackage{amsmath}
\usepackage{amsfonts}
\usepackage{xcolor}
\usepackage{fancyhdr}
\usepackage[margin=1.2in]{geometry}

\newcommand{\R}{\mathbb R}
\newcommand{\dha}{\dot H^1_a}
\newcommand{\bd}{\textbf{d}}
\newcommand{\eps}{\varepsilon}

\newcommand{\la}{\mathcal L_a}

\usepackage[symbol]{footmisc}
\renewcommand{\epsilon}{{\varepsilon}}

\numberwithin{equation}{section}
\newtheorem{theorem}{Theorem}[section]
\newtheorem{lemma}[theorem]{Lemma}

\newtheorem{definition}[theorem]{Definition}
\newtheorem{proposition}[theorem]{Proposition}

\newcommand{\qtq}[1]{\quad\text{#1}\quad}

\newcommand{\C}{\mathbb{C}}

\numberwithin{equation}{section}
\numberwithin{theorem}{section}
\numberwithin{figure}{section}
\numberwithin{table}{section}

\def\({\left(}
\def\){\right)}
\def\<{\left\langle}
\def\>{\right\rangle}

\def\eps{\varepsilon}

\title[Threshold scattering for energy critical  NLS$_a$]{Threshold Scattering for the Energy-Critical NLS with a Repulsive Inverse-Square Potential}

	\author[Z. Ma]{Zuyu Ma}
	\address[Z. Ma]{  The Graduate School of China Academy of Engineering Physics,
		Beijing 100088,\ P. R. China}
	\email{mazuyu23@gscaep.ac.cn}

    \author[Y. Song]{Yilin Song}
	\address[Y. Song]{The Graduate School of China Academy of Engineering Physics,
		Beijing 100088,\ P. R. China}
	\email{songyilin21@gscaep.ac.cn}

\author[K. Yang]{Kai Yang}
\address[K. Yang]{School of Mathematics, SouthEast University, Nanjing, P.R. China 211189}
\email{kaiyang@seu.edu.cn, yangkai99sk@gmail.com}
\thanks{KY was supported by the Jiangsu Shuang
Chuang Doctoral Plan and the Jiangsu Scientific Research Center of Applied Mathematics under Grant BK20233002.}

\author[X. Zhang]{Xiaoyi Zhang}
\address[X. Zhang]{Department of Mathematics, University of Iowa, Iowa City, IA 52242}
\email{xiaoyi-zhang@uiowa.edu}

\begin{document}
\maketitle 

\begin{abstract}

We study the threshold scattering problem for the energy-critical nonlinear Schrödinger equation with a repulsive inverse-square potential $\frac{a}{|x|^2} > 0$ in dimensions $d= 4, 5, 6$. On the energy level surface determined by the ground state of the energy-critical NLS without potential, we show that, despite the absence of a ground state in this setting, a strong form of rigidity persists below the kinetic threshold.
Specifically, we prove that any solution on this energy surface with kinetic energy strictly below that of the ground state is global and scatters to zero. Our approach combines refined modulation analysis, a center-translated global Virial estimate, and a bootstrap argument to control the modulation parameters.

\end{abstract}
\maketitle

\section{Introduction}
In this paper, we consider 
the initial value problem of the focusing NLS with an inverse square potential 
\begin{align} \label{NLSa}
\begin{cases}    (i\partial_t-\mathcal L_a)u= -|u|^\frac{4}{d-2}u, \\ 
u(t,x)|_{t=t_0}=u_0(x) \in \dot H^1(\R^d),
\end{cases}\tag{NLS$_{a}$}
\end{align}
where $u=u(t,x):\R^{ }\times\R^d\to \C$ and 
$$
    \mathcal{L}_a=-\Delta+\frac{a}{|x|^2}.
$$
For $a>-\left(\frac{d-2}{2}\right)^2$, the operator $\mathcal L_a$ is non-negative by sharp Hardy's inequality, and the quadratic form $\langle \mathcal L_a f, f\rangle$ defines a norm equivalent to the homogenous Sobolev space $\dot H^1$:
\[
\sqrt{\langle \mathcal L_a f, f\rangle}=\|\mathcal L_a^{\frac 12} f\|_{L_x^2}\sim \|f\|_{\dot H^1}.
\] 
We denote by $\dot H_a^1(\mathbb R^d)$ the space endowed with this equivalent norm. In particular, $\dot H_a^1$=$\dot H^1$ when $a=0$. 

\medskip 

\noindent Throughout this paper, solutions are understood in the following sense. 
\begin{definition}[Solution, \protect\cite{KMVZZ17}]
\label{D:solution} Let $t_{0}\in I\subset \mathbb{R}$ 
and $u_{0}\in \dot{H}_{a}^{1}(\R^d)$. A function $u:I\times \mathbb{R}%
^{d}\rightarrow \mathbb{C}$ is called a \emph{solution} to \eqref{NLSa} if 
$$u(t,x)\in
C_{t}\dot{H}_{a}^{1}\cap L_{t,x}^{\frac{2(d+2)}{d-2}}(J\times \mathbb R^d)
$$ for any compact interval $%
J\subset I$ and if $u$ satisfies the Duhamel formula 
\begin{equation}
u(t,x)=e^{-i(t-t_{0})\mathcal{L}_{a}}u_{0}+i\int_{t_{0}}^{t}e^{-i(t-s)\mathcal{%
L}_{a}}(|u(s,x)|^{\frac{4}{d-2}}u(s,x))\,ds,  \label{Duhamel}
\end{equation}
for all $t\in I$. The interval $I$ is called the lifespan of $u$. We say that $u$ is
a maximal-lifespan solution if it cannot be extended to a strictly larger interval. If $I=\mathbb{R}$, we call $u$ is a global solution.
\end{definition}
On their lifespan, solutions conserve energy:
\begin{align*}
E_a(u)&=\int_{\R^d}\Big(\frac{1}{2}|\nabla u|^2+\frac{a}{2|x|^2}|u|^2-\frac{d-2}{2d}|u|^\frac{2d}{d-2}\Big)\,dx=E_a(u_0).
\end{align*}
The equation \eqref{NLSa} is energy critical in the sense that the scaling 
$$u(t,x)\mapsto \lambda ^{\frac{2-d}{2}}u(\frac{t}{\lambda ^{2}},\frac{x}{\lambda })
$$ 
leaves the energy invariant.

\medskip

A central object in the analysis of \eqref{NLSa} is the ground state, defined as an optimizer of the sharp Sobolev inequality. 
\[
\|f\|_{L^{\frac{2d}{d-2}}}\le C_d \|f\|_{\dot H^1_a}.
\]
In the case $a\in (-\frac {(d-2)^2}4,0]$ or for $a> 0$ under a radial restriction, a standard variational argument shows that the optimizer is given by a positive, radial solution $W_a$ of  Euler-Largrange equation
\[
\mathcal{L}_a W_a = |W_a|^{\frac{4}{d-2}} W_a.
\]
An explicit formula for $W_a$ was obtained in \cite{Ter96, KMVZZ17}:
\begin{equation}\label{def:Wa}
W_{a}(x):= \big(d(d-2)\beta^2\big)^\frac{d-2}{4}\left(\frac{|x|^{\beta-1}}{1+|x|^{2\beta}}\right)^\frac{d-2}{2}\,, \beta=\sqrt{1+\tfrac{4a}{(d-2)^{2}}}. 
\end{equation}
When $a=0$, we simply write $W=W_0$. 

\medskip
The sharp Sobolev inequality takes the following form. 

\begin{proposition}[Sharp Sobolev embedding, \cite{KMVZZ17}]\label{P:Sob}
Let $d\ge 3$.  
\begin{enumerate}
\item[\rm(i)]  For all $f\in \dot H^1(\mathbb R^d)$ when $a\in( -\bigl(\tfrac{d-2}{2}\bigr)^{2},0]$ and all $f\in \dot H^1_{rad}(\mathbb R^d)$ when $a>0$, we have
\begin{equation}\label{ineq:Sob-}
\|f\|_{L^{\frac{2d}{d-2}}}
\leq\tfrac{\|W_{a}\|_{L^{\frac{2d}{d-2}}}}{\|W_{a}\|_{\dot H^{1}_{a}}}\,
\|f\|_{\dot H^{1}_{a}}.
\end{equation}
When $a=0$, the equality holds iff there exist $\alpha\in \mathbb C$, $\lambda>0$ and $x_0\in \mathbb R^d$ such that $f(x)=\alpha W_a(\lambda (x-x_0))$; in all other cases, the equality holds iff there exist $\alpha\in \mathbb C$, $\lambda>0$ such that $f(x)=\alpha W_a(\lambda x)$.

\item[\rm(ii)]  If $a>0$ and $f\neq 0$, then
\begin{equation}\label{ineq:Sob+}
\|f\|_{L^{\frac{2d}{d-2}}}
<\tfrac{\|W\|_{L^{\frac{2d}{d-2}}}}{\|W \|_{\dot H^1}} 
\|f\|_{\dot H^{1}_{a}}.
\end{equation}
In this case, equality never holds and the sharp constant cannot be improved.
\end{enumerate}
\end{proposition}

\medskip

 To place our results in context, we briefly review several known results for \eqref{NLSa}. 

\medskip

Although it may seem counterintuitive, the free case 
$(a=0)$ and the attractive potential case 
$(a<0)$ share many similarities, which we outline below. A fundamental result in this direction is the scattering/blow-up dichotomy below the ground state threshold, first established in the free case and subsequently extended to inverse-square potentials in a series of works \cite{KM06,KV10,Dodson19,Y20,Y21,YZ23}. For simplicity, we record here only the scattering component of these results. 
\begin{theorem}\label{scatter}
In either of the following two scenarios: 

\begin{itemize}
 \item  {$a=0$}. $ d\ge 4,\  u_0\in \dot H^1(\mathbb R^d)$; or $ d=3$, $u_0\in \dot H^1_{rad}(\mathbb R^d)$\\
 \item $ {a\in (-\big(\frac{d-2}{2}\big)^2+\big(\frac{d-2}{d+2}\big)^2,0)}$. $d=4, 5,6,\  u_0\in \dot H^1(\mathbb R^d)$; or $d=3$, $u_0\in \dot H_{rad}^1(\mathbb R^d)$. 
\end{itemize}
Solution $u$ to \eqref{NLSa} with initial data $u_0$ is certain to exist globally and scatter to $0$ if either of the following two estimates is satisfied
\begin{itemize}
\item $E_a(u_0)<E_a(W_a)$, $\|u_0\|_{\dot H^1_a}< \|W_a\|_{\dot H^1_a}$\\
\item $\sup_{t\in (-T_*, T^*)}\|u(t)\|_{\dot H^1_a}<\|W_a\|_{\dot H^1_a}. $
\end{itemize}
\end{theorem}

\medskip

 Having understood the scattering behavior below the ground-state energy level $E_a(u_a)<E_a(W_a)$, a natural and important question is to investigate solutions on the boundary of this energy cone, namely those satisfying 
\[
E_a(u_0)=E_a(W_a).
\] 
For this question, a series of studies —initiated by Duyckaerts and Merle \cite{DM08} in the free case and by Yang, Zeng and Zhang \cite{YZZ22} in the attractive potential case—has led to a classification of solutions at the threshold energy level into four distinct scenarios: scattering solutions, solitary waves, dynamics along stable/unstable manifolds of the ground state, and finite-time blow-up. See also \cite{CFR22, LZ09, SZ23,MMMZ25} for subsequence developments, including extensions to higher dimensions and to the non-radial setting.

\medskip

We now turn to the case $a>0$, where the results in \cite{Y20, Y21, YZ23} show that the scattering threshold is determined by the ground state in the free case. 
\begin{theorem}[\cite{Y20, Y21, YZ23}]\label{scatterpa} 
Let $a>0$. Let $u_0\in \dot H^1(\mathbb R^d)$ for $d=4,5,6$ and $u_0\in \dot H^1_{rad}(\mathbb R^d)$ for $d=3$. Suppose $u$ is the maximal life-span solution to \eqref{NLSa} on $(-T_*, T^*)$. Then $u$ must be global and scatter in both time directions if either of the following conditions is satisfied: 
\begin{itemize}
\item $\sup\limits_{t\in (-T_*, T^*)}\|u(t)\|_{\dot H^1_a}<\|W\|_{\dot H^1}$,
\item $E_a(u_0)<E_0(W),\  \|u_0\|_{\dot H^1_a}<\|W\|_{\dot H^1}$. 
\end{itemize}
\end{theorem}

\medskip

 In this paper, we investigate the behavior of solutions on the energy surface defined by 
\[
E_a(u)=E_0(W). 
\]
Due to the sharp Sobolev inequality, no ground state is attained at this energy level. This indicates that the dynamics on this surface may exhibit qualitatively different features from priori works. 

Our main result shows that, despite the absence of a ground state, a strong form of rigidity persists below a threshold. Specifically, we prove that any solution on this energy surface with subcritical kinetic energy is global and scatters to zero.

\begin{theorem}[Threshold scattering] \label{thm1}
Let $a>0$, $d\in \{4, 5, 6\}$, $u_0\in \dot H^1(\R^d)$ with 
\begin{align}\label{condsub}
E_a(u_0)=E_0(W),\qquad \| u_0\|_{\dot H_a^1}<\| W\|_{\dot  H^1(\R^d)}.
\end{align} 
Let $u:I\times\R^d\to\C$ be the maximal-lifespan solution to \eqref{NLSa} with $u(0)=u_0$. 
Then  $I=\mathbb R$ and $u$ satisfies 
\begin{align*}
\|u\|_{L_{t,x}^{\frac{2(d+2)}{d-2}}(\mathbb R\times \mathbb R^d)}<\infty,
\end{align*}
which further implies scattering: there exist $u_{\pm}\in \dot H^1(\mathbb R^d)$ such that
\[
\|u(t)-e^{-it\mathcal L_a}u_{\pm}\|_{\dot H^1}\to 0, \mbox{ as } \to \pm \infty. 
\]
\end{theorem}

\medskip

 This result provides a partial dynamical classification at the threshold energy level and shows that, even without a ground state manifold, solutions with subcritical kinetic energy exhibit purely dispersive behavior. In contrast, the dynamics of solutions with supercritical kinetic energy on the same energy surface remain open and appear to be substantially more intricate.

\medskip

In recent years, there has been significant progress on threshold dynamics for nonlinear Schrödinger equations in related settings. See, for instance, the study of the 3D cubic NLS outside a smooth convex obstacle by Duyckaerts, Landoulsi, and Roudenko \cite{DLR22}, as well as the work on the 3D cubic NLS with a repulsive potential by Miao, Murphy, and Zheng \cite{MMZ23}. In these settings--namely, inter-critical problems posed in the inhomogeneous Sobolev space $H^1$—the number of invariant symmetries is reduced, and the corresponding critical element is almost periodic with respect to translations only. Moreover, in the radial case when the translation symmetry is absent and the critical element is almost periodic solely with respect to scaling, an analogue threshold scattering has been established in \cite{HI25}.

By contrast, our problem is non-radial and energy-critical in nature, the almost periodicity of the critical element must incorporate both scaling and translation symmetries. This additional complexity constitutes the main technical challenge and requires  more delicate analysis.

\medskip

Our approach combines the concentration–compactness method with a refined modulation and a global Virial analysis adapted to the present setting. After establishing the variational structure on the energy surface--quantified by a distance function that measures the deviation of a solution from a ground state profile in the homogeneous Sobolev norm--a key ingredient is a delicate bootstrap scheme embedded within a global Virial analysis centered at an appropriately translated spatial point. This framework allows precise control of the modulation parameters simultaneously, and yields a crucial rigidity result, which rules out non-scattering solutions that remain close to the ground state manifold, ultimately leading to the proof of scattering. 

\medskip

Finally, we relate our work to recent developments on threshold solutions for inhomogeneous nonlinear Schrödinger equations \cite{CM23, CFM26, CFM26s, LYZ24} and other related dispersive models 
\cite{AHI24, AM24, DR10, KVZ12, MSZ24, MXZ13, YZZ25}.
Some of these results suggest a broader perspective in which the absence of a ground state does not preclude rigidity, but instead leads to new mechanisms governing the global dynamics.

 \medskip
The rest of the paper is organized as follows. In Section \ref{S:Pre}, we collect several properties of the operator $\la$, the local theory of \eqref{NLSa} and the coercivity of the energy.   In Section \ref{S:compact}, we prove the existence of an almost periodic solution assuming Theorem \ref{thm1} fails. In Section \ref{virial3}, we perform a modulation analysis for solutions near the ground state solution $W$. In Section \ref{S:rigidity}, we establish a rigidity result (Theorem \ref{rigidity} ) that precludes the existence of non-scattering solution that remains close to the ground state manifold for all time. Finally in Section \ref{S:reduction},
 we rule out all remaining possibilities, thereby completing the proof of the main theorem (Theorem \ref{thm1}).

\section{Preliminaries}\label{S:Pre}

 \textbf{Notations.} If $A$ and $B$ are non-negative quantities, we write $A\lesssim B$ to mean that $A\leq CB$ for some constant $C>0$. 
 
For the rest part of the paper, if not specified we restrict our consideration to $a>0$. 

\medskip
 
We use $\mathbf g_{\lambda_0, x_0}$ and $\mathbf{T}_{\lambda_0, x_0,t_0}$ to denote the symmetry transformation on spatial and temporal-spatial functions.
\begin{align}\label{nota1}
\mathbf g_{\lambda_0, x_0}(f)=\lambda_0^{-\frac{d-2} 2}f\bigl(\tfrac{x-x_0}{\lambda_0}\bigr), \quad \mathbf{T}_{\lambda_0, x_0,t_0} (v)=\mathbf g_{\lambda_0, x_0}(v(\tfrac{t}{\lambda_0^2}+t_0,x)).
\end{align}
In Section \ref{virial3}, we also use the following notation for symmetry transformation. 
\begin{align}\label{nota2}
 f_{[\theta,\mu,X]}=e^{i\theta}\mathbf g_{\mu,X} (f).  
\end{align}

\medskip
We use the following notations for the inner products between two complex-valued functions in $L^2$, $\dot H^1$ and $\dot H^1_a$:
\begin{align}\label{dotproduct}
\langle f,g\rangle_{L^2}=\Re\int f\bar g, \ \langle f,g\rangle_{\dot H^1}=\langle \nabla f,\nabla g\rangle_{L^2}, \ 
    \langle f,g\rangle_{\dot H_a^1}:=\langle \la f,g\rangle_{L^2}.
\end{align}

For $1<r<\infty$, we denote by $\dot H_a^{1,r}$ and $H_a^{1,r}$ the Sobolev spaces equipped with the norms:
\begin{align*}
\|f\|_{\dot H_a^{1,r}(\R^d)}=\|\sqrt{\la}f\|_{L^r(\R^d)},\,\, \|f\|_{ H_a^{1,r}(\R^d)}=\|\sqrt{1+\la}f\|_{L^r(\R^d)}.
\end{align*}

For a spacetime slab $I\times\R^d$, the norm in the spacetime Lebesgue space $L_t^q L_x^r(I\times\R^d)$ is defined by
\[
\|u\|_{L_t^{q}L_x^r(I\times\R^d)}:=\bigg(\int_I \|u(t)\|_{L_x^r(\R^d)}^q\,dt\bigg)^{1/q},
\]
with the usual modifications if $q$ or $r$ is infinity. When $q=r$, we abbreviate $L_t^qL_x^q=L_{t,x}^q$. Throughout this paper, we denote $2^*:=\frac{2d}{d-2}$ and abbreviate $\|f\|_{L^{\frac{2d}{d-2}}}$ by $\|f\|_{2^*}$. 

We will also use frequently the following spaces 
\begin{gather*}
    S(I):=L_{t,x}^{\frac{2(d+2)}{d-2}}(I\times\R^d),\\
    S^0(I)=L_t^2L_x^{2^*}(I\times\R^d)\cap L_t^\infty L_x^2(I\times\R^d),\\
    N^0(I)=L_t^1L_x^2(I\times \R^d)+ L_t^2L_x^{\frac{2d}{d+2}}(I\times\R^d),\\
    \dot{X}^{1}(I)=L_{t}^{\frac{2(d+2)}{d-2}}\dot{H}^{1,\frac{2d(d+2)}{d^{2}+4}}(I\times \mathbb{R}^{d}), \\
    \dot S^1(I)=\{f: \nabla f\in S^0(I)\},\quad  \dot N^1(I)=\{f: \nabla f\in N^0(I)\}.
\end{gather*}

We record the full range linear Strichartz estimate for Schr\"odinger operator $e^{it\la}$.
\begin{lemma}[Strichartz estimates,\cite{BPSTZ03, KT98, ZZ20}] Let $d\ge 3$  and $a>-\big(\frac{d-2}{2}\big)^2$. 
Then the solution $u:I\times\R^d\to\Bbb C$ to 
\begin{align*}
    (i\partial_t-\la)u=F(t,x)
\end{align*}
satisfies
\begin{align*}
    \|u\|_{S^0(I)}\lesssim\|u(t_0)\|_{L^2}+\|F\|_{N^0(I)},
\end{align*}
where $t_0\in I$.
\end{lemma}

We recall the stability lemma for the equation \eqref{NLSa}.
\begin{lemma}[Stability result,\cite{KMVZZ17}]\label{stability}
    Let $a>0$ and $I$ be a compact time interval. Suppose that $\tilde{u}$ is a near solution to \eqref{NLSa} on $I\times\R^d$ in the sense that
    \begin{align*}
        (i\partial_t-\la)\tilde{u}=-|\tilde{u}|^\frac{4}{d-2}\tilde u+e
    \end{align*}
     for some function $e$ and that it satisfies the bound:
         \begin{align*}
         \|\tilde{u}\|_{L_t^\infty \dot H_a^1(I\times\R^d)}\leq E<\infty,\,\,\|\tilde{u}\|_{S(I)}\leq L<\infty. 
     \end{align*}
     Let $t_0\in I$, $u_0\in\dot H_a^1(\R^d)$, and assume the smallness condition
     \begin{align*}
         \|u_0-\tilde{u}(t_0)\|_{\dot H_a^1}+\|\la^\frac{1}{2}e\|_{N^0(I)}\leq \varepsilon
     \end{align*}
      for some $0<\varepsilon<\varepsilon_1:=\varepsilon_1(E,L)$. 
      
      Then, there exists a unique solution to \eqref{NLSa} with initial data $u_0$ at $t=t_0$ satisfying
      \begin{gather*}
\big\|\la^\frac12u\big\|_{S^0(I)}\leq C(E,L),\\
\big\|\la^\frac12(u-\tilde{u})\big\|_{S^0(I)}\leq C(E,L)\varepsilon.
      \end{gather*}
\end{lemma}
We now recall linear profile decompositions relative to Sobolev inequality and Strichartz estimate of $e^{-it\la}$, respectively.

\begin{lemma}[$\dot{H}^1_a$ linear profile decomposition I, \cite{YZZ22}]\label{LPDa}
 Let $\{f_n\}$ be a bounded sequence in $\dot{H}^1_a(\mathbb{R}^d)$. After passing to a subsequence, there exist $J^* \in \{0, 1, 2, \ldots\} \cup \{\infty\}$, $\{\phi^j\}_{j=1}^{J^*} \subset \dot{H}^1_a(\mathbb{R}^d)$, $\{(\lambda_n^j, x_n^j)\}_{j=1}^{J^*} \subset \mathbb{R}^+ \times \mathbb{R}^d$ such that for every $0 \leq J \leq J^*$, we have the decomposition
\[
f_n = \sum_{j=1}^J \phi_n^j + r_n^J, \quad \phi_n^j = \mathbf{g}_{\lambda_n^j,x_n^j}(\phi^j), \quad r_n^J \in \dot{H}^1_a(\mathbb{R}^d)
\] 
satisfying
\begin{eqnarray*}
&&\lim_{J \to J^*} \limsup_{n \to \infty} \|r_n^J\|_{2^*} = 0,\\
&&\lim_{n \to \infty} \left( \|f_n\|_{\dot{H}^1_a}^2 - \sum_{j=1}^J \|\phi_n^j\|_{\dot{H}^1_a}^2 - \|r_n^J\|_{\dot{H}^1_a}^2 \right) = 0,\\
&&\lim_{n \to \infty} \left( \|f_n\|_{2^*} ^{2^*} - \sum_{j=1}^J \|\phi_n^j\|_{2^*} ^{2^*} - \|r_n^J\|_{2^*} ^{2^*}  \right) = 0.
\end{eqnarray*}

Moreover, for all $j \neq k$, we have the asymptotic orthogonality property

\[
\lim_{n \to \infty} \left( \left| \frac{\lambda_n^j}{\lambda_n^k} \right| + \left| \frac{\lambda_n^k}{\lambda_n^j} \right| + \frac{|x_n^j - x_n^k|^2}{\lambda_n^j \lambda_n^k} \right) = \infty.
\]

Finally, we may also assume for each $j$, either $|x_n^j|/\lambda_n^j \to \infty$ or $x_n^j \equiv 0$ and
\[
\|\phi_n^j\|_{\dot{H}^1_a} \to \|\phi^j\|_{X_j} = 
\begin{cases} 
\|\phi^j\|_{\dot{H}^1} & \text{as } \frac{|x_n^j|}{\lambda_n^j} \to \infty, \\
\|\phi^j\|_{\dot{H}^1_a} & \text{as } x_n^j \equiv 0.
\end{cases} \tag{A.7}
\]
\end{lemma}

\begin{lemma}[$\dot{H}^1_a$ linear profile decomposition II, \cite{KMVZZ17}]\label{LPDb}
 Let $\{f_n\}$ be a bounded sequence in $\dot{H}^1_a(\mathbb{R}^d)$. After passing to a subsequence, there exist $J^* \in \{0,1,2,\ldots,\infty\}$, $\{\phi^j\}_{j=1}^{J^*} \subset \dot{H}^1_a$, $\{\lambda_n^j\}_{j=1}^{J^*} \subset (0,\infty)$, $\{x_n^j\}_{j=1}^{J^*} \subset \mathbb{R}^d$ and $\{t_n^j\}_{j=1}^{J^*} \subset \mathbb{R}$ such that for each finite $0 \leq J \leq J^*$, we have the decomposition
\[
f_n = \sum_{j=1}^J \phi_n^j + w_n^J \quad \text{with} \quad \phi_n^j = \mathbf{g}_{\lambda_n^j,x_n^j}(e^{-it_n^j \mathcal{L}_a^{n^j}} \phi^j) \quad \text{and} \quad w_n^J \in \dot{H}_a^1
\]
satisfying
\begin{gather}
\lim_{J \to J^*} \limsup_{n \to \infty} \|e^{-it \mathcal{L}_a} w_n^J\|_{S(\mathbb{R})} = 0,\label{remainder}\\
\lim_{n \to \infty} \left\{ \|f_n\|_{\dot{H}_a^1}^2 - \sum_{j=1}^J \|\phi_n^j\|_{\dot{H}_a^1}^2 - \|w_n^J\|_{\dot{H}_a^1}^2 \right\} = 0,\label{orthogonal-1}\\
\lim_{n \to \infty} \left\{ \|f_n\|_{2^*}^{2^*} - \sum_{j=1}^J \|\phi_n^j\|_{2^*}^{2^*} - \|w_n^J\|_{2^*}^{2^*} \right\} = 0.\label{orthogonal-2}
\end{gather}

Here $\mathcal{L}_a^{n^j} = -\Delta + \frac{a}{|x + y_n^j|^2}$ with $y_n^j = \frac{x_n^j}{\lambda_n^j}$. Moreover, for all $j \neq k$, we have the asymptotic orthogonality property
\[
\lim_{n\to\infty}\; \left| \frac{\lambda_n^j}{\lambda_n^k} \right| + \left| \frac{\lambda_n^k}{\lambda_n^j} \right| + \frac{|t_n^j (\lambda_n^j)^2 - t_n^k (\lambda_n^k)^2|}{\lambda_n^j \lambda_n^k} + \frac{|x_n^j - x_n^k|^2}{\lambda_n^j \lambda_n^k} = \infty .
\]
Furthermore, we may assume that for each $j$: either $t_n^j \equiv 0$ or $t_n^j \to \pm \infty$; either $\frac{|x_n^j|}{\lambda_n^j} \to \infty$ or $x_n^j \equiv 0$.
\end{lemma}

\medskip

Using the sharp Sobolev inequality, we  show the positivity of $E_a(u)$ as well as the coercivity of kinetic energy on the energy level surface.

\begin{lemma}[Coercivity of energy]\label{energy-control}
Assume $u_0\neq 0$ satisfies $\|u_0\|_{\dot H^1_a}\le \|W\|_{\dot H^1}$. Then 
\begin{align}\label{Ec1}
 \  \frac{\|u_0\|_{\dot H_a^1}^2}{\|W\|_{\dot H^1}^2}< \frac{E_a(u_0)}{E_0(W)}.
\end{align}
If in addition $E_a(u_0)=E_0(W)$, then the solution $u(t)$ to \eqref{NLSa} with initial data $u_0$ satisfies 
\begin{align}\label{Ec2}
\|u(t)\|_{\dot H^1_a}<\|W\|_{\dot H^1},
\end{align}
for all time $t$ of existence. 
\end{lemma}

\begin{proof} From sharp Sobolev embedding \eqref{ineq:Sob+}, we obtain
\begin{align*}
E_a(u_0)
&=\frac12 {\|u_0\|_{\dot H_a^1}^2}-\frac 1{2^*}{\|u_0\|_{2^*}^{2^*}}
>\frac 12 {\|u_0\|_{\dot H_a^1}^2- \frac 1{2^*}\bigl(\tfrac{\|W\|_{2^*}}{\|W \|_{\dot H^{1}}}\|u_0\|_{\dot{H}_a^1}\bigr)^{2^*}}\\
& = \tfrac{\|u_0\|_{\dot H_a^1}^2}{\|W\|_{\dot H^1}^2}\bigl(\frac 12\|W\|_{\dot H^1}^2 - \frac 1{2^*}\|W\|_{2^*}^{2^*}\bigl(\tfrac{\|u_0\|_{\dot H_a^1}}{\|W \|_{\dot H^1}}\bigr)^{2^*-2}\bigr) \ge \tfrac{\|u_0\|_{\dot H_a^1}^2}{\|W\|_{\dot H^1}^2}E_0(W)>0.
\end{align*}
\eqref{Ec1} is proved, hence the positivity of $E_a(u_0)$. 

If in addition $E_a(u_0)=E_0(W)$, \eqref{Ec1} immediately implies that $\|u_0\|_{\dot H^1_a}<\|W\|_{\dot H^1}$. The estimate for all time $t$ of existence simply follows from the fact that the set $\{t: \|u(t)\|_{\dot H_a^1}<\|W\|_{\dot H^1}\}$ is both open and closed in view of \eqref{Ec1} and the continuous dependence on the norm.
\end{proof}

\medskip

We recall a result from \cite{Y20} on the construction of nonlinear profiles when $\frac{|x_n|}{\lambda_n}\to \infty$, which will be used in \textit{type (ii) nonlinear profile} in the proof of Theorem \ref{compact}.
\begin{lemma}[\cite{Y20}]\label{ENP}
Let $d=4,5,6$ and let $(t_n,\lambda_n,x_n)\in \mathbb R\times \mathbb R^+\times \mathbb R^d$ satisfying 
\[
t_n\equiv 0 \mbox{ or } {t_n\to \pm \infty}, \ \frac{|x_n|}{\lambda_n}\to \infty \mbox{ as } n\to \infty. 
\]
 Let $\phi \in \dot{H}^1(\mathbb{R}^d)$ satisfy \[
\begin{cases}
E_0(\phi) < E_0(W), \, \|\phi\|_{\dot{H}^1} < \|W\|_{\dot{H}^1}, & \text{if } t_n \equiv 0, \\
\frac{1}{2}\|\phi\|_{\dot{H}^1}^2 < E_0(W), \, \|\phi\|_{\dot{H}^1} < \|W\|_{\dot{H}^1}, & \text{if } t_n \to \pm\infty.
\end{cases}
\]
Denote
\[
\phi_n(x) =\mathbf{g}_{\lambda_n,x_n}(e^{-it_n \mathcal{L}_a^n} \phi),
\]
where $\mathcal{L}_a^n=-\Delta+\frac{a}{|x+y_n|^2}$ and $y_n = x_n \lambda_n^{-1}$. 
Then for sufficiently large $n$, there exists a global solution $v_n$ to \eqref{NLSa} with $v_n(0) = \phi_n$ such that 
\[
\|v_n\|_{S(\mathbb{R})} \lesssim 1.
\]
Furthermore, given $\varepsilon >0$, $\exists $ $N\in \mathbb{N}$ and $\psi _{\varepsilon
}\in C_{c}^{\infty }(\mathbb{R}\times \mathbb{R}^{d})$ such that for all $%
n\geq N$, 
\begin{equation}
\Vert v_{n}(t,x)-\mathbf{T}_{\lambda_n,x_n,t_n}(\psi_{\epsilon})\Vert _{\dot{X}^{1}(\mathbb{R})}<\varepsilon.  \label{enp2}
\end{equation}
\end{lemma}

\section{Compactness of non-scattering solution}\label{S:compact}

In this section, we prove the almost periodicity for solutions violating Theorem~\ref{thm1}. 

\medskip
We argue by contradiction. Suppose Theorem~\ref{thm1} fails, then there must exist a solution $
u(t,x): (-T_*, T^*)\times \mathbb R^d\to \mathbb C
$ to \eqref{NLSa}
satisfying 
\begin{gather}
    E_a(u_0)=E_0(W),\,\,\|u_0\|_{\dot H_a^1}<\|W\|_{\dot H^1},\label{sub-threshold-1}\\
\|u\|_{S((-T_*,T^*))}=\infty.\notag
\end{gather}
Without loss of generality, we may assume $u$ does not scattering forward in time:
\begin{align}\label{non-scattering}
\|u\|_{S([0,T^*) )}=\infty. 
\end{align}
Then the following precompactness result holds.
\begin{theorem}\label{compact}
For solution $u$ of \eqref{NLSa} obeying \eqref{sub-threshold-1} and \eqref{non-scattering}, there exist 
\[
(\lambda(t), x(t)): (0, T^*)\to \mathbb R^+\times \mathbb R^d
\]
such that 
\begin{align*}
\Big\{\lambda(t)^{-\frac{d-2}{2}}u(t,\frac{x-x(t)}{\lambda(t)}), \  t\in (0, T^*)\Big\} \mbox{ is precompact in }\dot H^1(\R^d).
\end{align*}
Consequently, there exist $(\lambda(t),x(t))\in C(I; \R^+\times \R^d)$ and function $C(\eta)$ such that for $t\in I$ and $\eta>0$, the following holds
\begin{align*}
    \int_{|x-x(t)|\geq C(\eta)\lambda(t)}|\nabla u(t,x)|^2+\frac{|u(t,x)|^2}{|x|^2}+|u(t,x)|^{\frac{2d}{d-2}}\,dx\leq \eta.
\end{align*}
\end{theorem}

\begin{proof} As the general framework for obtaining the precompactness (modulo symmetires) of critical element is by now standard, we only sketch the main steps.

 We need to show that for any sequences $\{\tau_n\}\subset [0,T^*)$, there exists a sequence of scaling parameters $\{\lambda_n\}$ such that the following holds after passing to a subsequence,
\begin{align}\label{compactness-property}
    \lambda_n^{-\frac{d-2}{2}}u\Big(\tau_n,\frac{x-x_n}{\lambda_n}\Big)\mbox{  converges strongly in }\dot H^1(\R^d).
\end{align}
By the continuity of $u$ in $\dot H^1$ with respect to time, it suffices to consider $\tau_n\to T^*$.  

From the coercivity of energy (Lemma \ref{energy-control}), it follows that  $u(\tau_n)$ is bounded in $\dot{H}^1(\R^d)$, and we  apply the linear profile decomposition (Lemma \ref{LPDb}) to get
\begin{align*}
u_n:=u(\tau_n)=\sum_{j=1}^J\phi_n^j+w_n^J, 
\end{align*}
such that for $0\leq J\leq J^*$, 
\begin{align}
&\lim_{J \to J^*} \limsup_{n \to \infty} \|e^{-it \mathcal{L}_a} w_n^J\|_{S(\mathbb{R})} = 0,\label{remind}\\
&\lim_{n\to\infty}E_a(u_n)=\lim_{n\to\infty}\big(\sum_{j=1}^JE_a(\phi_n^j)+E_a(w_n^J)\big)=E_0(W),\label{energy-ortho}
\\
&\lim_{n\to\infty}\|u_n\|_{\dot H_a^1}^2=\lim_{n\to\infty}\Big(\sum_{j=1}^J\|\phi_n^j\|_{\dot H_a^1}^2+\|w_n^J\|_{\dot H_a^1}^2\Big)\le \|W\|_{\dot H^1}^2,\label{kinetic-ortho}
\end{align}
where the last inequality follows from \eqref{Ec2} in Lemma \ref{energy-control}.

\medskip

Next, we consider three possible cases for $J^*$.

\medskip

\noindent\textbf{Case 1: $J^*=0$.}
In this case, by using \eqref{remainder} in Lemma \ref{LPDb}, we obtain
\begin{align*}
\limsup_{n \to \infty} \big\|e^{-it\la}u(\tau_n)\big\|_{S([0,+\infty))}=0.
\end{align*}
By the stability result (Lemma \ref{stability}) and some standard estimates, it follows that  $u(t+\tau_n)$ can be well approximated by $e^{-it\la}u(\tau_n)$, hence,
\begin{align*}
\|u(t+\tau_n)\|_{S([0,+\infty))}=\|u(t)\|_{S([\tau_n,+\infty))} \lesssim1,
\end{align*}
for all sufficiently large $n$, which clearly contradicts \eqref{non-scattering}. This rules out Case 1.

\medskip

\noindent\textbf{Case 2: $J^*\ge 2$.}
 Following \cite{KMVZZ17,YZ23}, we construct a nonlinear profile corresponding to each linear profile $\phi_n^j$. 
  
\textit{Type (i) nonlinear profile:} $x_n^j \equiv 0$. 

If $t_n^j \equiv 0$, let $v^j : I_j \times \mathbb{R}^d \to \mathbb{C}$ be the maximal-lifespan solution to \eqref{NLSa} with initial data $\phi^j$. If $t_n^j \to \pm\infty$, define $v^j : I_j \times \mathbb{R}^d \to \mathbb{C}$ to be the maximal-lifespan solution to (\eqref{NLSa}) that scatters to $e^{-it\mathcal{L}_a} \phi^j$ as $t \to \pm\infty$. Set
\[
v_n^j(t, x) = \mathbf{T}_{\lambda_n^j,x_n^j,t_n^j}(v^j),
\]
which is a maximal-lifespan solution on $I_n^j = \{t \in \mathbb{R} \mid (\lambda_n^j)^{-2}t + t_n^j \in I_j\}$.

\textit{Type (ii) nonlinear profile:} $\frac{|x_n^j|}{\lambda_n^j} \to \infty$. 

In this case, $\Vert\phi_n^j\Vert_{\dot{H}_a^1}\to \Vert\phi^j\Vert_{\dot{H}^1}$. We define $v_n^j$ to be the maximal-lifespan solution of (\eqref{NLSa}) on $I_n^j$ with $v_n^j(0) = \phi_n^j$. The existence of a global scattering solution $v_n^j$ follows from Lemma \ref{ENP}.

From the above construction, we have
\begin{equation}
\lim_{n \to \infty} \|v_n^j(0) - \phi_n^j\|_{\dot{H}^1} = 0. \tag{65}
\end{equation}
Using $a> 0$ and the coercivity of energy (Lemma \ref{energy-control}), we obtain
\begin{align}
  \|\phi^j\|_{\dot H^1}^2< \|\phi_n^j\|_{\dot H_a^1}^2<\frac{\|W\|_{\dot H^1}^2}{E_0(W)}E_a(\phi_n^j),\quad \|w_n^J\|_{\dot H_a^1}^2< \frac{\|W\|_{\dot H^1}^2}{E_0(W)}E_a(w_n^J).\label{profile-control}
\end{align}
This implies that $$\liminf\limits_{n\to\infty}E_a(\phi_n^j)>0,\quad E_a(w_n^J)>0,$$ and hence, for all $1\le j\le J^*$,  we have 
$$
0<E_a(\phi_n^j)<E_0(W), \quad \|\phi_n^j\|_{\dot H^1}<\|W\|_{\dot H^1}.
$$

By the stability result Lemma \ref{stability},  Theorem \ref{scatterpa} (for type (i)) and Lemma \ref{ENP} (for type (ii)), each $v_{n}^{j}$ is a global solution satisfying
 \[
 \Vert v_{n}^{j}\Vert _{S(\mathbb{R})}<\infty, \mbox{ and }\Vert v_{n}^{j}\Vert _{\dot{X}^{1}(\mathbb{R})}<\infty.
 \]
By density argument, for any $\varepsilon >0$, there exists $\psi
_{\varepsilon }^{j}\in C_{c}^{\infty }$ such that%
\begin{equation*}
\Vert v^{j}-\psi _{\varepsilon }^{j}\Vert _{\dot{X}^{1}(\mathbb{R}%
)}<\varepsilon .
\end{equation*}%
A change of variable yields
\begin{equation}
\Vert v_{n}^{j}(t,x)-\mathbf{T}_{\lambda_n^j,x_n^j,t_n^j}(\psi _{\varepsilon }^{j})\Vert _{\dot{X}%
^{1}(\mathbb{R})}<\varepsilon.  \label{close}
\end{equation}%

We record two standard properties (cf. \cite{KMVZZ17, Y20, Y21, YZ23}):

\begin{itemize}
\item\textit{ Boundedness:}
\begin{equation}
\Vert v_{n}^{j}\Vert _{\dot{X}^{1}(\mathbb{R})}\lesssim \left\{ 
\begin{array}{c}
\Vert \phi _{n}^{j}\Vert _{\dot{H}_{a}^{1}},\text{ \ \ \ \ if }\Vert \phi
_{n}^{j}\Vert _{\dot{H}_{a}^{1}}\leq \eta _{0}; \\ 
1,\text{\ \ \ \ \ \ \ \ \ \ \ \ \ if }\Vert \phi _{n}^{j}\Vert _{\dot{H}_{a}^{1}}>\eta _{0};
\end{array}
\right.  \label{Bnd_v}
\end{equation}
for some small $\eta_0>0$.
 
\item \textit{Orthogonality (for $j\neq k)$:}
\begin{equation}
\Vert \nabla v_{n}^{j}\nabla v_{n}^{k}\Vert _{\frac{d+2}{d-2},\frac{d(d+2)}{%
d^{2}+4}}+\Vert v_{n}^{j}v_{n}^{k}\Vert _{\frac{d+2}{d-2},\frac{d+2}{d-2}%
}+\Vert \nabla v_{n}^{j}v_{n}^{k}\Vert _{\frac{d+2}{d-2},\frac{d(d+2)}{%
d^{2}-d+2}}=o_{n}(1).  \label{disjoint1}
\end{equation}
\end{itemize}

We now construct the approximate solution 
\begin{align}\label{uapp}
u_n^J(t,x)=\sum_{j=1}^Jv_n^j(t,x)+e^{-it\la}w_n^J(x).
\end{align}
Arguing as in \cite{Y21}, we obtain:

\begin{itemize}

\item \textit{Claim 1: }$\lim_{n\rightarrow \infty }\left\Vert
u_{n}^{J}(0)-u_{n}(0)\right\Vert _{\dot{H}_{a}^{1}}=0$ for any $J$.

\item\textit{Claim 2: }$\overline{\lim\limits_{n\rightarrow \infty }}\left\Vert
u_{n}^{J}\right\Vert _{\dot{X}^{1}(\mathbb{R})}\lesssim _{E_{c},\delta }1$
uniformly in $J$.

\item\textit{Claim 3:} $\lim_{J\rightarrow \infty }\overline{\lim }_{n\rightarrow
\infty }\Vert (i\partial _{t}+\mathcal{L}_{a})u_{n}^{J}+|u_{n}^{J}|^{%
\frac{4}{d-2}}u_{n}^{J}\Vert _{\dot{N}^{1}(\mathbb{R})}=0$.
\end{itemize}

Thus, for sufficiently large $n$ and $J$, $u_{n}^{J} $ is an approximate solution to \eqref{NLSa} with finite scattering norm. By the stability result (Lemma \ref{stability}), it follows that $u_{n}$
is also global with finite scattering norm, contradicting \eqref{non-scattering}. This excludes Case 2.

\medskip

\noindent\textbf{Case 3: $J^*=1$.}
The linear profile decomposition reduces to 
\begin{equation}\label{zz1}
    u(\tau_n)=\phi_n+w_n, \quad \phi_n=\lambda_n^{-\frac{d-2}{2}}(e^{-it_n\la^n}\phi)\big(\tfrac{x-x_n}{\lambda_n}\big).
\end{equation}
Clearly, the desired precompactness result follows provided that $\|w_n\|_{\dot{H}^1}\to 0$ as $ n\to\infty$ and $t_n\equiv 0$.

We first show that 
\begin{equation}\label{zz2}
    \|w_n\|_{\dot{H}^1}\to 0\; (n\to \infty).
\end{equation}
If this were not the case, then, arguing as in Case 2, one could construct an approximate solution as in \eqref{uapp} (with $J=1$ ) having finite  scattering norm. By the stability result, this would imply that $u$ is  global scattering solution, contradicting \eqref{non-scattering}. 

It remains to rule out the possibility that  $t_n\to\pm\infty$ as $n\to\infty$. We consider only the case $t_n\to\infty$, as the case $t_n\to -\infty$ can be handled similarly. For completeness, we outline the steps following \cite{YZ23}. 

\medskip

\textsf{\bf Case $t_n\to+\infty$ and $x_n\equiv0$.} 
By  \eqref{zz1} and \eqref{zz2}, together with the monotone convergence theorem and Strichartz estimate, we have
\begin{align*}
\|e^{-it\la}u(\tau_n)\|_{S(\{t\ge 0\})}&\lesssim  \|e^{-it\la}\phi_n\|_{S(\{t\ge 0\})}+ \|w_n\|_{\dot{H}^1}\\
&=\|\lambda_n^{-\frac{d-2}{2}} e^{-i(\lambda_n^{-2}t+t_n)\mathcal{L}_a}\phi(\tfrac{x}{\lambda_n})\|_{S(\{t\ge 0\})}+ \|w_n\|_{\dot{H}^1}\\
&=\|e^{-it\la}\phi\|_{S(\{t\ge t_n\})}+\|w_n\|_{\dot{H}^1}
 \to0,\,\,\mbox{as }n\to\infty.
\end{align*}
By the small data theory for (\eqref{NLSa}), it follows that $u(t+\tau_n)$ scatters forward in time for sufficiently large $n$, contradicting \eqref{non-scattering}.

\medskip

\textsf{\bf Case $t_n\to+\infty$ and $\frac{|x_n|}{\lambda_n}\to\infty$.} 
Arguing similarly, we obtain
\begin{align*}
\|e^{-it\la}u(\tau_n)\|_{S(\{t\ge 0\})}&\lesssim  \|e^{-it\la}\phi_n\|_{S(\{t\ge 0\})}+ \|w_n\|_{\dot{H}^1} =\|e^{-it\la^n}\phi\|_{S(\{t\ge t_n\})}+\|w_n\|_{\dot{H}^1} \\
&\le \|e^{-it\la^n}\phi-e^{it\Delta}\phi\|_{S(\{t\ge t_n\})}+\|e^{it\Delta}\phi\|_{S(\{t\ge t_n\})}+\|w_n\|_{\dot{H}^1} 
 \to0,
\end{align*}
where $\|e^{it\Delta}\phi\|_{S(\{t\ge t_n\})}\to0$  follows from the monotone convergence theorem and the Strichartz estimate, and $\|e^{-it\la^n}\phi-e^{it\Delta}\phi\|_{S(\{t\ge t_n\})}\to 0$ follows from Corollary 3.5 in \cite{KMVZZ17}.
\end{proof}

 Next, we recall a result from \cite{DM08, MMMZ25}, which asserts global existence for any sub-threshold solution. With minor modifications, the same proof applies in our setting. 
\begin{proposition}\label{blowup}
    Suppose that $u$ is a solution to \eqref{NLSa} in its maximal lifespan $I$ and satisfies
    \begin{align}\label{sub-threshold-con}
    E_a(u_0)=E_0(W),\,\,\|u_0\|_{\dot H_a^1 }\leq\|W\|_{\dot H^1 },
    \end{align}
    then $I=\R$.
\end{proposition}

Define the distance function $\bd(f)$ by
\begin{align}\label{dist}
\bd(f)=\bigl | \|f\|_{\dha}^2-\|W\|_{\dot H^1}^2\bigr|. 
\end{align}
Before concluding this section, we collect two results concerning on convergence of $\bd(u(t))$ along time sequences. The first is an immediate corollary from Theorem \ref{scatterpa}. 

\begin{lemma}\label{sequencec}
There exists a sequence $\{t_n\}\subset \R$ with $t_n\to\infty$ such that $\bd(u(t_n))\to0$ as $n\to\infty$,  where $u$ is the solution to \eqref{NLSa} given in Theorem \ref{compact}.
\end{lemma}
\begin{proof}
 We argue by contradiction. Suppose there exists $c_1>0$ such that $\bd(u(t))\geq c_1$ for all $t\geq0$. Then, for some $\varepsilon>0$, we have
\begin{align*}
    \sup_{t\in[0,\sup I)}\|u\|_{\dot H_a^1}<(1-\varepsilon)\|W\|_{\dot H^1},
\end{align*}
where $I$ is the maximal lifespan of $u$. By applying Theorem \ref{scatterpa} forward in time, we conclude that $\|u\|_{S([0,\infty))}<\infty$, which contradicts \eqref{non-scattering}.
\end{proof}

Next we establish a conditional convergence result for $\bd(u(t))$. 

\begin{lemma}\label{distant}
 Let $u$ be the solution to \eqref{NLSa} given in Theorem \ref{compact}.  Let $t_n\to\infty$ be any sequence  such that $|x(t_n)|\lambda(t_n)^{-1}\to\infty$. Then $\bd(u(t_n))\to 0$. 
\end{lemma}
\begin{proof}
We argue by contradiction. Suppose the conclusion fails. Then there exists a sequence $t_n\to\infty$ such that $|x(t_n)|\lambda(t_n)^{-1}\to\infty$ and 
$$\bd(u(t_n))\geq c$$ 
for some $c>0$. By Theorem \ref{compact} (after passing to a subsequence if necessary), there exists $\phi\in\dot{H^1}(\R^d)$ such that 
\begin{align*}
    u(t_n)=\lambda(t_n)^{-\frac{d-2}{2}}\phi\big(\tfrac{x-x(t_n)}{\lambda(t_n)}\big)+w_n=\phi_n+w_n,
\end{align*}
with
\begin{align}\label{decompose-2}
\big\|w_n\big\|_{\dot H_a^1(\R^d)}\to0,\quad\text{as }n\to\infty.
\end{align}
This convergence, together with $\bd(u(t_n))\geq c$ implies
\begin{align*}
    E_0(\phi)=E_0(W),\qquad \|\phi\|_{\dot H^1}<\|W\|_{\dot H^1}.
\end{align*}
Let $w$ be the solution to NLS (i.e., no potential) with initial data $w(0)=\phi$. By \cite[Theorem 1.3]{MMMZ25}, we have  
\begin{align}\label{wfail}
   \mbox{either }  \|w \|_{S([0,\infty))}<\infty,\,\,\mbox{or }     \|w \|_{S((-\infty,0])}<\infty.
\end{align}

\medskip 

\noindent \textbf{Case I}: $\|w \|_{S([0,\infty))}<\infty$. 

By applying essentially the same proof forward in time as in Lemma \ref{ENP} (see also \cite[Theorem 3.1]{YZ23}),  there exists a forward global solution $v_n$ to $(\text{NLS}_a)$ with $v_n(0) = \phi_n$ such that
$$
\|v_n\|_{S([0,\infty))}\lesssim 1
$$
uniformly for all large $n$. By the stability result (Lemma \ref{stability}), it follows that
$$
\|u\|_{S([t_n,\infty))}\lesssim 1
$$
uniformly for all large $n$, which contradicts \eqref{non-scattering} in Theorem \ref{compact}.

\medskip

\noindent\textbf{ Case II}: $\|w \|_{S([-\infty,0))}<\infty$. 

Similarly, by applying the backward-in-time version of the argument in Lemma \ref{ENP} (see also  \cite[Theorem 3.1]{YZ23}),  there exists a backward global solution $v_n$ to \eqref{NLSa} with $v_n(0) = \phi_n$ such that
$$
\|v_n\|_{S((-\infty,0])}\lesssim 1
$$
 uniformly for all large $n$. By the stability result (Lemma \ref{stability}), we then obtain 
 $$
\|u\|_{S((-\infty,t_n])}\lesssim 1
$$
 uniformly for all large $n$, again contradicting \eqref{non-scattering}.

In either case, we arrive at a contradiction, which completes the proof.
\end{proof}

\section{Modulation analysis}\label{virial3} 
This section is devoted to the modulation analysis, which provides a dynamical framework for tracking solutions near the ground state manifold. A key outcome of this analysis is precise control of the modulation parameters in terms of the distance function. This control plays a crucial role in subsequent global analysis, as has been demonstrated in many previous works. 

Let $\bd(f)$ be the distance function defined in \eqref{dist}. 
We will show that $\bd(f)$ “measures" the distance between $f$ and the ground state manifold as made precise in the following result.  

\begin{proposition}[Variational characterization of $W$]\label{P:vcW}
Assume that $f\in \dot{H}_{a}^{1}(\R^d)$ and $%
E_a(f)=E_0(W)$. Then for any $\varepsilon>0$, there exists $\delta>0$ such that
whenever  $\bd(f)<\delta, $ we have
$$
\inf_{\theta\in\mathbb S^1, \mu>0,X
\in \R^d}\| e^{i\theta}\mathbf g_{\mu,X} (f) -W\|_{\dot H^1}<\varepsilon. 
$$
\end{proposition}

\begin{proof}
We argue by contradiction. Suppose the conclusion fails. Then there exist $\varepsilon _{0}>0$ and a sequence $\{f_n\}$ bounded in $\dot H^1$ such that 
\begin{align}
\text{ }E_a(f_{n})=E_0(W),\text{ }\mathbf{d}
(f_{n})\to 0\; (n\to\infty),  \label{vcW 1} \\
\inf_{\theta \in \mathbb{S}^1,\mu >0,X\in \R^d}\Vert  e^{i\theta}\mathbf g_{\mu,X} (f_n) -W\Vert _{\dot{H}^{1}}>\varepsilon _{0}.  \label{vcW 2}
\end{align}
By rescaling $f_n$ as
\[
f_n\to f_n\cdot\frac{\|W\|_{\dot H^1}}{\|f_n\|_{\dha}},
\] we may assume
\begin{align}\label{guan}
\|f_n\|_{\dha}=\|W\|_{\dot H^1}, \ \|f_n\|_{2^*}\to \|W\|_{2^*}, \ \inf_{\theta \in \mathbb{S}^1,\mu >0,X\in \R^d}\Vert  e^{i\theta}\mathbf g_{\mu,X} (f_n) -W\Vert _{\dot{H}^{1}}>\varepsilon _{0}.
\end{align}
Applying Lemma \ref{LPDa} to $\{f_n\}$, we extract a subsequence in $f_n$ (still denoted by $f_n$) and obtain the decomposition
$$
f_n=\textstyle\sum\limits_{j=1}^J\phi_n^j +r_n^J, 
$$
for each $J\in \{1, \cdots, J^*\}$ with the stated properties. In particular, 
by the $\dot H_a^1$ decoupling in Lemma \ref{LPDa} and \eqref{guan}, we have
\begin{align}\label{93}
\|W\|_{\dot H^1}^2= \lim_{n\to \infty}\biggl(\textstyle\sum\limits_{j=1}^J\|\phi_n^j\|_{\dha}^2+\|r_n^J\|_{\dha}^2\biggr)=\textstyle\sum\limits_{j=1}^J\|\phi^j\|_{X^j}^2+ \lim\limits_{n\to \infty}\|r_n^J\|_{\dha}^2,
\end{align}
where $\|\cdot\|_{X^j}=\|\cdot\|_{\dha}$ if $x_n^j\equiv 0$ and $\|\cdot\|_{X^j}=\|\cdot\|_{\dot H^1}$ if $\frac{|x_n^j|}{\lambda_n^j}\to \infty$.  Consequently,
\begin{align}\label{lbd}
\textstyle\sum\limits_{j=1}^{J^*} \|\phi^j\|_{X^j}^2
\le \|W\|_{\dot H^1}^2=\|W\|_{2^*}^{2^*}. 
\end{align}

On the other hand, by decoupling in $L^{2^*}(\mathbb{R}^d)$, \eqref{guan} and the sharp Sobolev embedding, we obtain
\begin{align*}
\|W\|_{2^*}^{2^*}=\lim_{n\to \infty}\|f_n\|_{2^*}^{2^*}=\textstyle\sum\limits_{j=1}^{J^*}\|\phi^j\|_{2^*}^{2^*}\le \textstyle\sum\limits_{j=1}^{J^*}\|\phi^j\|_{\dot H^1}^{2^*}\cdot \tfrac{\|W\|_{2^*}^{2^*}}{\|W\|_{\dot H^1}^{2^*}},
\end{align*}
which implies
\begin{align}\label{second}
\|W\|_{\dot H^1}^{2^*}\le \textstyle\sum\limits_{j=1}^{J^*}\|\phi^j\|_{\dot H^1}^{2^*}.
\end{align}
Combining \eqref{lbd} and \eqref{second}, we deduce  
\begin{align*}
(\textstyle\sum\limits_{j=1}^{J^*} \|\phi^j\|_{X^j}^2
)^{\frac{2^*}{2}}\le \textstyle\sum\limits_{j=1}^{J^*}\|\phi^j \|_{\dot H^1}^{2^*}. 
\end{align*}
Since $a>0$ implies $\|\phi^j\|_{\dot H^1}<\|\phi^j\|_{\dha}$, it follows that 
\[
J^*=1, \quad  \tfrac{|x_n^1|}{\lambda_n^1} \to \infty, \text{ and } \limsup_{n\to \infty} \|r_n^1\|_{2^*} =0. 
\]
Then \eqref{lbd} and \eqref{second} yield 
$$\|\phi^1\|_{\dot H^1}=\|W\|_{\dot H^1}, \ \|\phi^1\|_{2^*}=\|W\|_{2^*}.$$
 Hence $\phi^1=e^{i \theta_0}\mathbf{g}_{\mu_0, x_0}(W)$ for some $\theta_0, \mu_0, x_0$. Moreover, by \eqref{93}, 
$$
f_n=(\lambda_n^1)^{-\frac 12}\phi^1\bigl(\tfrac {x-x_n^1}{\lambda_n^1}\bigr)+r_n^1, \mbox{  \it{and} } \|r_n^1\|_{\dot H^1}\to 0 \; (n\to \infty).
$$
This contradicts the last inequality in \eqref{guan}, completing the proof. 
\end{proof}

Before introducing the modulation decomposition, we recall a coercivity property of a quadratic form.

Expanding the energy functional $E_0$ at the ground state $W$ yields: 
\begin{equation}
E_0(W+g)=E_0(W)+Q(g)+O(\left\Vert g\right\Vert _{\dot{H}^{1}}^{3}+\Vert g \Vert _{\dot{H}^{1}}^{2^*}). \label{Ed}
\end{equation}
Here the quadratic form $Q(g)$ is associated with the linearized operator around $W$ and can be written explicitly as 
\[
Q(g)=\frac{1}{2}\left\Vert g\right\Vert _{\dot{H}^{1}}^{2}-\frac{1}{2}
\int W^{\frac 4{d-2}}\biggl(\tfrac{d+2}{d-2}(\mathbf{Re}g)^{2}+(\mathbf{Im}g)^{2}\biggr).
\]
The following result establishes the coercivity of $Q$ on a co-dimension $d+3$ subspace orthogonal to the negative and null directions of the linearized operator, namely $\{W, iW, W_1, \partial_1 W, \cdots, \partial_dW\}$, where $W_1=\frac{d-2}{2}W+x\cdot \nabla W$.  For a related result in the presence of an inverse square potential, we refer to \cite{YZZ22}; the proof there can be adapted to handle the potential-free case.

\begin{lemma}[Coercivity of quadratic form $Q$, \cite{DM08, Rey}]
\label{L:Q} Let $g\in \dot{H}^{1}(\R^d)$ satisfy 
\[g\perp_{\dot H^1} \{W,iW,W_1,\partial_1 W,\cdots ,\partial_d W\},
\] 
with orthogonality taken with respect to $\dot H^1$-inner product \eqref{dotproduct}. Then there exist constants $c, C>0$ such that
\begin{equation*}
c \left\Vert g\right\Vert _{\dot{H}^{1}}^{2}\leq Q(g)\leq C \left\Vert g\right\Vert _{\dot{H}^{1}}^{2}.
\end{equation*}
\end{lemma}

For notational convenience, throughout the remainder of this section we write
\begin{align*}
 \mathcal{H}=\{ iW,W_1,\partial_1 W,\cdots ,\partial_d W\}.
\end{align*}
and recall the notation $ f_{[\theta,\mu,X]}=e^{i\theta}\mathbf g_{\mu,X} (f)$.

Proposition \ref{P:vcW} and Lemma \ref{L:Q}, together with the implicit function theorem, yield the following modulation result. 
\begin{lemma}
\label{L:ift}Let $d\geq 3$. There exist $\delta _{0}>0$, $\eps_0>0$ such that
for any $f\in \dot{H}^{1}(\R^d)$ with $E_a(f)=E(W)$, $\mathbf{d}(f)<\delta _{0}$,
there exists unique triple $(\theta ,\mu ,X)$ in $\mathbb S^1\times\mathbb R^+\times\mathbb R^d$
 such that 
\begin{align}\label{ift1}
f_{[\theta ,\mu ,X]}\bot_{\dot H^1} \mathcal{H},  \mbox{ and } \|f_{[\theta ,\mu ,X]}-W\|_{\dot H^1}<\eps_0. 
\end{align} 
 Moreover, the mapping $f\mapsto (\theta ,\mu ,X)$ is $C^{1}$. The decomposition 
 \begin{align}\label{ift2}
f_{[\theta,\mu,X]} = W + v = W + \alpha W + h,\  h\bot_{\dot H^1}\{W, iW, W_1, \partial_1 W, \cdots, \partial _d W\} 
\end{align}
 satisfies
\begin{align}\label{ift3}
\Vert v\Vert_{\dot H^1}\sim |\alpha|+\Vert h\Vert_{\dot H^1}\sim \|h\|_{\dot H^1}+ \left(\int_{\R^d} \frac{a}{|x|^2}|f|^2\right)^{\frac{1}{2}}\sim |\alpha|\sim \bd(f). 
\end{align}
\end{lemma}

\begin{proof} We divide the proof into two steps. The uniqueness of parameters follows from the argument in \cite{YZZ22}, and we omit the details. 

\textbf{Step 1.} Define functionals $J_{k}(k=1,\cdots ,d+2)$ on $\mathbb S^1\times \mathbb R^+\times \mathbb R^d$ by
\begin{eqnarray*}
J_{1}(\theta ,\mu ,X,f) &=&\left\langle f_{[\theta ,\mu ,X]},iW\right\rangle
_{\dot{H}^{1}},\text{ \ \ }J_{2}(\theta ,\mu ,X,f)=\left\langle f_{[\theta
,\mu ,X]},W_1\right\rangle _{\dot{H}^{1}}; \\
J_{k}(\theta ,\mu ,X,f) &=&\left\langle f_{[\theta ,\mu
,X]},\partial_{k-2} W\right\rangle _{\dot{H}^{1}},\text{ \ \ }k=3,\cdots ,d+2.
\end{eqnarray*}%
By construction, $J_{k}(0,1,\overrightarrow{0},W)=0$ for all $k$. A direct computation shows that the Jacobian matrix at $(0,1,\overrightarrow{0},W)$ is invertible; in fact, only the diagonal entries are nonzero:
\begin{eqnarray*}
\frac{\partial J_{1}}{\partial \theta }(0,1,\overrightarrow{0},W)
&=&\left\Vert W\right\Vert _{\dot{H}^{1}}^{2},\text{ \ \ \ }\frac{\partial
J_{2}}{\partial \mu }(0,1,\overrightarrow{0},W)=-\left\Vert W_1%
\right\Vert _{\dot{H}^{1}}^{2} \\
\frac{\partial J_{k}}{\partial X_{k-2}}(0,1,\overrightarrow{0},W)
&=&\left\Vert \partial_{k-2} W \right\Vert _{\dot{H}^{1}}^{2},\text{ \ }k=3,\cdots
,d+2.
\end{eqnarray*}%
Hence, by the implicit Function Theorem, there exist $\varepsilon _{0}$,$\eta
_{0}>0$ such that for any $f\in \dot{H}^{1}(\R^d)$ with $\left\Vert f-W\right\Vert
_{\dot{H}^{1}}<\varepsilon _{0}$, there exists a unique $(\theta ,\mu
,X)$ with $|\theta |+|\mu -1|+|X|\leq \eta _{0}$ such that $f_{[\theta
,\mu ,X]}\bot \mathcal{H}$ in $\dot{H}^{1}(\R^d)$.

By Proposition \ref{P:vcW}, for this $\varepsilon _{0}>0$ there exists $%
 \delta _{0}>0$ such that whenever $\mathbf{d}(f)<\delta _{0}$, one can find $(\theta _{0},\mu _{0},X_{0})$ with
\begin{equation*}
\Vert f_{[\theta _{0},\mu _{0},X_{0}]}-W\Vert _{\dot{H}^{1}}<\varepsilon
_{0}.
\end{equation*}%
Applying the implicit function theorem at this point yields \eqref{ift1} and the $C^1$ dependence. 

\textbf{Step 2.} 
Using the decomposition \eqref{ift2} and expanding $E_0$ as in \eqref{Ed} we obtain 
\begin{align*}
& E_a(f) = E_0(f) + \frac{1}{2}\int_{\R^d} \frac{a}{|x|^2}|f|^2 dx=E_0(f_{[\theta,\mu,X]})+ \frac{1}{2}\int_{\R^d} \frac{a}{|x|^2}|f|^2 dx\\
&= E_0(W) +  \alpha^2 Q(W) + Q(h) + \frac{1}{2}\int_{\R^d} \frac{a}{|x|^2}|f|^2 dx+ O(|\alpha|^{\min(3, 2^*)} + \|h\|^{\min(3, 2^*)}_{\dot H^1}).
\end{align*}
Since $Q(W)<0$ and $E_a(f)=E_0(W)$,  Lemma \ref{L:Q} implies
\[
 \|h\|^2_{\dot H^1}+ \int_{\R^d} \frac{a}{|x|^2}|f|^2 dx \sim - \alpha^2 Q(W) + O(|\alpha|^{\min(3,2^*)} + \|h\|^{\min(3, 2^*)}_{\dot H^1}). 
\]
For $\delta_0$ sufficiently small, the error terms are negligible and hence,
$$
 \|h\|^2_{\dot H^1}+ \int_{\R^d} \frac{a}{|x|^2}|f|^2 dx \sim \alpha^2.
$$
Using the definition of $\bd (f)$, we compute
\[
\bd (f) = |\|W\|^2_{\dot H^1} - \|f\|^2_{\dot H^1_a}| = \bigl|(2\alpha + \alpha^2)\|W\|_{\dot H^1}^2 + \|h\|_{\dot H^1}^2 +\int_{\R^d} \frac{a}{|x|^2}|f|^2 dx\bigr |\sim |\alpha|.
\]
Finally, 
$$\Vert v\Vert_{\dot H^1}\sim |\alpha|+\Vert h\Vert_{\dot H^1}\sim |\alpha|,$$
which establishes \eqref{ift3}. 
\end{proof}

Let $u$ be a solution of \eqref{NLSa} on a time interval $I$ such that 
\[E_a(u_{0})=E(W), \ \mathbf{d}(u(t))<\delta _{0}, \ \forall t\in I.
\] From Lemma \ref{L:ift}, there exist $\theta (t),\mu (t),X(t)$
such that the following decomposition holds
\begin{equation}
u_{[\theta (t),\mu (t),X(t)]}(t)=W+v(t)=(1+\alpha (t))W+h(t),
\ h(t)\perp_{\dot H^1} \{W\}\cup \mathcal H. 
\label{md}
\end{equation}%

We have the following estimates on these modulation functions.

The approaches in previous works such as \cite{DM08, MMMZ25, YZZ22}, which rely on a self-similar change of variables, appear to be more involved when handling the inverse-square potential. To avoid these complications, we instead derive the difference equation and carry out the necessary estimates directly, without resorting to a self-similar transformation.
 
\begin{lemma}[Modulation]
\label{L:mod}Let $u(t)$ be the solution mentioned above satisfying the decomposition \eqref{md}. Then for all $t\in I$, we have
\begin{align}\label{mod1} 
\Vert v(t)\Vert_{\dot H^1}\sim |\alpha(t)|+\Vert h(t)\Vert_{\dot H^1}\sim \|h(t)\|_{\dot H^1}+ \left(\int_{\R^d} \tfrac{|u(t)|^2}{|x|^2} dx\right)^{\frac{1}{2}}\sim |\alpha(t)|\sim \bd(u(t)),
\end{align}
\begin{align}
 \left|\frac{\alpha_t}{\mu^2}\right|+\left|\frac{\theta_t}{\mu^2}\right|+\left|\frac{\mu_t}{\mu^3}\right|+\left|\frac{X_t}{\mu^2}-\frac{\mu_t X}{\mu^3}\right| \lesssim \bd(u(t)).  \label{mod2}
\end{align}
Here, the implicit constants do not depend on $t$.
\end{lemma}
\begin{proof}
 \eqref{mod1} follows directly from Lemma \ref{L:ift}, thus it remains to prove \eqref{mod2}. For simplicity we drop $t$ from the modulation functions.

Let $\rho (t,x)=u_{[\theta ,\mu ,X]}(t)$, or equivalently  $u(t,x)=e^{-i\theta }\mu ^{\frac{d-2}{2}}\rho (t,\mu x+X))$. Then $\rho$ satisfies
\begin{equation*}
i\frac{\rho _{t}}{\mu ^{2}}+\Delta \rho +|\rho |^{\frac{4}{d-2}}\rho +
\frac{\theta _{t}}{\mu ^{2}}\rho +i\frac{\mu _{t}}{\mu ^{3}}\frac{d-2}{2}
\rho +i\nabla \rho \cdot (\frac{\mu _{t}}{\mu ^{2}}x+\frac{X_{t}}{\mu^2 })-\frac{a}{\mu^2|x|^2}\rho=0.
\end{equation*}
With the new variable $y=\mu x+X$, \eqref{md} can be written as
\[
\rho(t,y)=W(y)+v(t,y).
\]
This together with the equation for W: 
\[
\quad-\Delta W=W^{\frac{d+2}{d-2}},\ \frac{d-2}{2}W+y\cdot \nabla
W=W_1,
\]
gives the following equation of $v(t,y)$:
\begin{align*}
\frac{\partial _{t}v}{\mu^2}+\mathcal{L}(v)+R(v)-i\frac{\theta_t}{\mu^2}(W+v)+\frac{\mu _{t}}{\mu^3}
W_1+\frac{\mu _{t}}{\mu^3 }(\tfrac{d-2}{2}v+y\cdot \nabla v) \\
+ \nabla (W+v)\cdot (\frac{X_t}{\mu^2}-\frac{\mu_t X}{\mu^3})+i\frac{a }{|y-X|^2}\rho=0,
\end{align*}
 where \begin{align*}
 &\mathcal{L}(v)=(\Delta+W^{\frac{4}{d-2}})v_2+i(-\Delta-\tfrac{d+2}{d-2} W^{\tfrac{4}{d-2}})v_1, \\
&R(v)=-i|W+v|^{\frac{4}{d-2}}(W+v)+iW^{\frac{d+2}{d-2}}+i\tfrac{d+2}{d-2} W^{\tfrac{4}{d-2}}v_{1}-W^{\tfrac{4}{d-2}}v_{2} .
 \end{align*}
Inserting the relation $v=h+\alpha (t)W$ and collecting higher order terms to the right-hand side, we obtain the following equation
\begin{align}
&\frac{\partial_t h}{\mu^2} + \frac{\alpha_t}{\mu^2} W + (\Delta + W^{\frac{4}{d-2}}) h_2  + i \left( -\Delta - \tfrac{d+2}{d-2} W^{\frac{4}{d-2}} \right) h_1 \notag \\
&\quad - i \alpha \frac{4}{d-2} W^{\frac{d+2}{d-2}}- i \frac{\theta_t}{\mu^2} W + \frac{\mu_t}{\mu^3} W_1 + \nabla W \cdot \left(  \frac{X_t}{\mu^2}-\frac{\mu_t}{\mu^3} X \right)+  i \frac{a}{|y - X|^2} \rho \notag \\
&= -R(v) + i \frac{\theta_t}{\mu^2} v - \frac{\mu_t}{\mu^3} \left( \tfrac{d-2}{2} v + \nabla y \cdot \nabla v \right) - \nabla v \cdot \left(\frac{X_t}{\mu^2}-\frac{\mu_t X}{\mu^3} \right) \label{mod3}.
\end{align}

Taking inner product in $\dot H^1(\R^d)$ and applying \eqref{mod1}, we have
\begin{eqnarray*}
&&|\langle \eqref{mod3},W\rangle _{\dot{H}^{1}}|+|\langle
\eqref{mod3},iW\rangle _{\dot{H}^{1}}|+|\langle \eqref{mod3},
W_1\rangle _{\dot{H}^{1}}|+\sum_{k=1}^{d}|\langle \eqref{mod3},\partial_k W\rangle _{\dot{H}^{1}}|  \lesssim \mathcal{E}(t), 
\end{eqnarray*}
where $$\mathcal{E}(t)=\mathbf{d}(u(t))\left( \mathbf{d}(u(t))^{\min(1,\frac{4}{d-2})}+\left|\frac{\theta _{t}}{\mu^2}\right|+\left\vert 
\frac{\mu _{t}}{\mu^3 }\right\vert +\left|\frac{X_t}{\mu^2}-\frac{\mu_t X}{\mu^3}\right|\right). $$ 
 Recall Lemma 5.6 in \cite{DM08} that
\begin{equation*}
\left\Vert R(v)\right\Vert _{L^{\frac{2d}{d+2}}}\lesssim \left\Vert
v\right\Vert _{\dot{H}^{1}}^{2}+\left\Vert v\right\Vert _{\dot{H}^{1}}^{\frac{d+2}{d-2}}\lesssim \bd(u(t))^{\min (2,\frac{d+2}{d-2})}.
\end{equation*}
Taking the dot product between the last expanded equation and $W$, $iW$, $W_1$, $\partial_k W(k=1,\cdots ,d)$ in $\dot{H}^{1}(\R^d)$ separately, we get 
\begin{gather*}
\frac{\alpha _{t}}{\mu^2}\left\Vert W\right\Vert _{\dot{H}^{1}}^{2}=\langle (-\Delta
-W^{\frac{4}{d-2}})h_{2},W\rangle _{\dot{H}^{1}}+\langle - i \tfrac{a}{|y - X|^2} \rho , W\rangle_{\dot H^1} +O(\mathcal{E}(t))
\\ 
\frac{\theta _{t}}{\mu^2}\left\Vert W\right\Vert _{\dot{H}^{1}}^{2}=\langle (-\Delta
-W^{\frac{4}{d-2}})h_{1},W\rangle _{\dot{H}^{1}}-\alpha
 \frac{4}{d-2} \langle W^{\frac{d+2}{d-2}},W\rangle _{\dot{H}^{1}}+\langle  i \tfrac{a}{|y - X|^2} \rho , iW\rangle_{\dot H^1}+O(\mathcal{E}
(t)) \\ 
\frac{\mu _{t}}{\mu^3 }\left\Vert W_1\right\Vert _{\dot{H}
^{1}}^{2}=\langle (-\Delta -W^{\frac{4}{d-2}})h_{2},W_1
\rangle _{\dot{H}^{1}}+\langle - i \tfrac{a}{|y - X|^2} \rho , W_1 \rangle_{\dot H^1}+O(\mathcal{E}(t)) \\ 
\left(\frac{X_t}{\mu^2}-\frac{\mu_t X}{\mu^3} \right)\left\Vert \partial_k W\right\Vert _{\dot{H}^{1}}^{2}=\langle (-\Delta -W^{\frac{4}{d-2}})h_{2},\partial_k W\rangle _{\dot{H}^{1}}+\langle - i \tfrac{a}{|y - X|^2} \rho , \partial_k W\rangle_{\dot H^1}+O(\mathcal{E}(t)).
\end{gather*}

Here we only explain the computation involving the inverse square potential part; the rest computations are similar to the classical NLS case \cite{DM08, MMMZ25}. For instance, from \eqref{mod1}, H\"older and Hardy inequality, we can estimate
\begin{align*}
 &|\langle - i \tfrac{a}{|y - X|^2} \rho , W\rangle_{\dot H^1}|=|\langle  i \tfrac{a}{|y - X|^2} u_{[\theta ,\mu ,X]}  , \Delta W\rangle |\lesssim \Vert\tfrac{ u_{[\theta ,\mu ,X]}}{|y-X|} \Vert_2  \Vert\tfrac{\Delta W}{|y-X|} \Vert_2 \\
=& (\int_{\R^d} \tfrac{|u(x)|^2}{|x|^2} dx)^{\frac{1}{2}} \left\Vert\tfrac{W^{\frac{d+2}{d-2}}}{|y-X|} \right\Vert_2 \lesssim \bd (u(t)) \Vert W^{\frac{d+2}{d-2}} \Vert_{\dot H^1} \lesssim \bd(u(t)).
\end{align*}Similarly, we have
$$
|\langle - i \tfrac{a}{|y - X|^2} \rho , i W\rangle_{\dot H^1}|+|\langle - i \tfrac{a}{|y - X|^2} \rho , W_1\rangle_{\dot H^1}|+|\langle - i \tfrac{a}{|y - X|^2} \rho , \partial_k W\rangle_{\dot H^1}|\lesssim \bd(u(t)).
$$
Therefore,
$$
\left|\frac{\alpha_t}{\mu^2}\right|+\left|\frac{\theta_t}{\mu^2}\right|+\left|\frac{\mu_t}{\mu^3}\right|+\left|\frac{X_t}{\mu^2}-\frac{\mu_t X}{\mu^3}\right| \lesssim \bd(u(t))+O(\mathcal{E}(t)).
$$
 By taking $\delta_0$ small enough, we conclude \eqref{mod2}.
\end{proof}

\section{A rigidity result}\label{S:rigidity}
In this section, we prove a rigidity result that rules out any sub-threshold solution which remain close to the ground state manifold all the time and fails to scatter. By Proposition \ref{blowup}, any sub-threshold solution is global. Thus we consider a solution $u$ to \eqref{NLSa} satisfying 
\begin{align}\label{344}
E_a(u)=E_0(W), \ \|u\|_{\dot H^1_a}<\|W\|_{\dot H^1}, \ \|u\|_{S([0,\infty))}=\infty. 
\end{align}
Our goal is to prove the following

\begin{theorem}\label{rigidity}
    There does not exist a solution to \eqref{NLSa} satisfying \eqref{344} and 
    \begin{align}\label{347}
   \mathbf{d}(u(t))\leq \delta_0 
   \end{align}
   for all $t\ge 0$, provided $\delta_0$ is sufficiently small. 
\end{theorem}
The remainder of this section is devoted to the proof of this result. Applying Lemma \ref{L:mod} to $u$, we obtain $C^1$ functions 
\begin{equation*}
\theta : [0,\infty)\rightarrow \R,\quad x:  [0,\infty)\rightarrow \mathbb{R}^d ,\quad \mu: [0,\infty)\rightarrow \mathbb{R}_+, \quad \alpha: [0,\infty)\rightarrow \mathbb{R}_+
\end{equation*}
such that the decomposition 
\begin{align}\label{decompz0}
&u_{[\theta(t),\mu(t),x(t)]}(t)-W=v(t)=\alpha(t)W+h(t)=u_{[\theta(t),\mu(t),x(t)]}(t)-W, \quad h\perp_{\dot H^1} W
\end{align}
holds, together with the estimates in \eqref{mod1}, \eqref{mod2}. 
Note here for notational simplicity, we write $x(t)$ for the translation parameter (previously denoted by $X(t)$).

The key step in proving Theorem \ref{rigidity} is to obtain precise control on scaling and translation functions via a delicate bootstrap argument, as stated below. 
\begin{proposition}\label{bootstrap}  Let $u$ be the solution described above. Then there exists a constant   $C_2>0$ such that, if
$
|x(0)|\ge C_2
$
and $\delta_0$ is sufficiently small, the following global estimates hold: 
\begin{equation}\label{bootse1}
\frac{2}{3}\leq\left|\frac{\mu(t)}{\mu(0)}\right|\leq  \frac{3}{2}\qtq{and} \left|\frac{x(t)}{\mu(t)}-\frac{x(0)}{\mu(0)}\right|\leq \frac{1}{2\mu(0)},\ \forall t\in [0,\infty).
\end{equation}
\end{proposition}

\begin{proof} By a standard continuity argument, it suffices to establish a bootstrap improvement. Namely, suppose that on some time interval $[0,T]$ we have the a priori bounds
\begin{equation}\label{bootstrap1}
\frac{1}{2}\leq\left|\frac{\mu(t)}{\mu(0)}\right|\leq 2\qtq{and} \left|\frac{x(t)}{\mu(t)}-\frac{x(0)}{\mu(0)}\right|\leq \frac{1}{\mu(0)},\ \forall t\in [0,T].  
\end{equation}
We will show that these bounds can be improved to: 
\begin{equation}\label{boostrap2}
\frac{2}{3}\leq\left|\frac{\mu(t)}{\mu(0)}\right|\leq \frac{3}{2}\qtq{and} \left|\frac{x(t)}{\mu(t)}-\frac{x(0)}{\mu(0)}\right|\leq \frac{1}{2\mu(0)},\ \forall t\in [0,T].
\end{equation}

The refined control \eqref{boostrap2} on both scaling and translation parameters can be obtained through a global Virial analysis with a spatially translated center. 

Let $\phi(x)$ be a smooth radial function such that 
\begin{equation*}
\phi (x)=
\begin{cases}
|x|^{2}, & |x|\leq 1, \\ 
0, & |x|>2,
\end{cases}
\qquad  \phi''(r)\le 2,
\end{equation*}
and  define $\phi _{R}(x)=R^{2}\phi \bigl(\frac{x}{R}\bigr)$.
We then introduce the truncated Virial quantity
\begin{equation*}
V_{R}(t)=\int_{\mathbb{R}^{d}}\phi _{R}(x)\Big|u\big(t,x-\frac{x(0)}{\mu(0)}\big)\Big|^{2}\,dx.
\end{equation*}%
A direct computation  (see also \cite{YZZ22}) yields
\begin{align*}
& \partial _{t}V_{R}(t) =2\mathbf{Im}\int_{\mathbb{R}^{d}}\overline{u\big(t,x-\tfrac{x(0)}{\mu(0)}\big)}\,\nabla u\big(t,x-\tfrac{x(0)}{\mu(0)}\big)\cdot \nabla \phi _{R}dx;  
\end{align*}
\begin{align*}
 &\partial _{tt}V_{R}(t)\notag
\\ = &4\mathbf{Re}\int_{\mathbb{R}^{d}}(\phi
_{R})_{jk}(x)u_j\big(t,x-\tfrac{x(0)}{\mu(0)}\big)(t)\bar{u}_k\big(t,x-\tfrac{x(0)}{\mu(0)}\big)(t)\,dx-\frac{4}{d}\int_{\mathbb{R}^{d}}(\Delta \phi _{R})\big|u(t,x-\tfrac{x(0)}{\mu(0)})\big|^{2^*}\,dx  \notag \\
 \quad &-\int_{\mathbb{R}^{d}}(\Delta^2 \phi
_{R})\big|u(t,x-\tfrac{x(0)}{\mu(0)})\big|^{2}\,dx+4a\int_{\mathbb{R}^{d}}\tfrac{x-\tfrac{x(0)}{\mu(0)}}{|x-\tfrac{x(0)}{\mu(0)}|^{4}}\cdot \nabla \phi
_{R}\big|u(t,x-\tfrac{x(0)}{\mu(0)})\big|^{2}\,dx  \notag \\
 = &8(\Vert u(t)\Vert _{\dot{H}_a^{1}}^{2}-\Vert u(t)\Vert_{2^*}^{2^*})+4a\int_{\mathbb{R}^{d}}\big[\tfrac{x-\frac{x(0)}{\mu(0)}}{|x-\tfrac{x(0)}{\mu(0)}|^{4}}\cdot \nabla \phi
_{R}-\tfrac{2|x-\tfrac{x(0)}{\mu(0)}|^2}{|x-\tfrac{x(0)}{\mu(0)}|^4}\big]\Big|u\big(t,x-\tfrac{x(0)}{\mu(0)}\big)\Big|^{2}\,dx\notag\\
\quad &+A_R^0\Big(u\big(t,x-\tfrac{x(0)}{\mu(0)}\big)\Big)\\
= & \frac{16}{d-2}(\Vert W\Vert _{\dot{H}^{1}}^{2}-\Vert u(t)\Vert _{\dot{H}_a^{1}}^{2})+4a\int_{\mathbb{R}^{d}}\big[\tfrac{x-\frac{x(0)}{\mu(0)}}{|x-\tfrac{x(0)}{\mu(0)}|^{4}}\cdot \nabla \phi
_{R}-\tfrac{2|x-\tfrac{x(0)}{\mu(0)}|^2}{|x-\tfrac{x(0)}{\mu(0)}|^4}\big]\Big|u\big(t,x-\tfrac{x(0)}{\mu(0)}\big)\Big|^{2}\,dx\notag\\
\quad &+ A_R^0\Big(u\big(t,x-\tfrac{x(0)}{\mu(0)}\big)\Big),
\notag \end{align*}
where $A_R^0$ collects the remaining error terms (excluding the $4a$-term), namely
\begin{eqnarray*}
A_{R}^0(u(t)) &=&\int_{|x|>R}\big(\frac{4\partial _{r}\phi _{R}}{r}-8\big)|\nabla
u(t)|^{2}+(8-\frac{4}{d}\Delta \phi _{R})|u(t)|^{2^*}dx \\
&&-\int_{|x|>R}\Delta ^{2}\phi _{R}|u(t)|^{2}+\frac{4(r\partial _{rr}\phi _{R}-\partial _{r}\phi _{R})}{r^{3}}|x\cdot \nabla u(t)|^{2}dx.
\end{eqnarray*}
In the following, we denote
\[
\tilde W=e^{-i\theta(t)}\mathbf{g}_{\mu(t),x(t)}^{-1} W. 
\]
Since $A_R^0$ coincides with the error term in truncated Virial identity in the free case, results in \cite{DM08, MMMZ25} imply 
\begin{equation} \label{ARW0}
A_R^0\Big(\tilde{W}\big(x-\frac{x(0)}{\mu(0)}\big)\Big)=0.
\end{equation} 
Hence, we can rewrite
\begin{align}\label{V2}
\partial_{tt}V_R(t)
&=\frac{16}{d-2} \ \bd(u(t))+4a\int_{\mathbb{R}^{d}}\big[\tfrac{x-\frac{x(0)}{\mu(0)}}{|x-\frac{x(0)}{\mu(0)}|^{4}}\cdot \nabla \phi
_{R}-\tfrac{2|x-\frac{x(0)}{\mu(0)}|^2}{|x-\frac{x(0)}{\mu(0)}|^4}\big]\Big|u\big(t,x-\frac{x(0)}{\mu(0)}\big)\Big|^{2}\,dx\notag\\
&+A_R^0\Big(u\big(t,x-\frac{x(0)}{\mu(0)}\big)\Big)-A_R^0\Big(\tilde{W}\big(x-\frac{x(0)}{\mu(0)}\big)\Big).
\end{align}

Next we establish two estimates for the error terms appearing above by appropriately choosing $R$ and invoking the bootstrap hypothesis \eqref{bootstrap1}. Let $C_1=C_1(a,d)>0$ be a constant to be determined later and let $\delta_0$ be sufficiently small. Set  
\[
R=\frac{C_1}{\mu(0)}.
\]
Then 
 \begin{equation}\label{bst1}
    \Big|A_R^0\Big(u\big(t,x-\frac{x(0)}{\mu(0)}\big)\Big)-A_R^0\Big(\tilde{W}\big(x-\frac{x(0)}{\mu(0)}\big)\Big)\Big|\leq \frac{4}{d-2}\bd(u(t)),\, \forall t\in[0,T]
\end{equation}
Moreover, if in addition $|x(0)|\geq C_2=C_2(a, d, C_1)$, then
\begin{equation}\label{bst2}
\Bigg|4a\int_{\mathbb{R}^{d}}\Big[\tfrac{x-\frac{x(0)}{\mu(0)}}{|x-\frac{x(0)}{\mu(0)}|^{4}}\cdot \nabla \phi
_{R}-\tfrac{2|x-\frac{x(0)}{\mu(0)}|^2}{|x-\frac{x(0)}{\mu(0)}|^4}\Big]\Big|u(t,x-\tfrac{x(0)}{\mu(0)})\Big|^{2}\,dx\Bigg|\leq \frac{4}{d-2}\bd(u(t,x)),\ \forall t\in[0,T].
\end{equation} 
Assuming \eqref{bst1} and \eqref{bst2} for the moment, we derive the improved bound   \eqref{boostrap2}, as shown below.

Indeed, substituting \eqref{bst1} and \eqref{bst2} into \eqref{V2} yields
\[
\partial_{tt}V_{R}\geq \frac{8}{d-2}\bd(u(t)). 
\]
A directly computation together with \eqref{mod1} yields
\begin{align}
    |\partial_{t}V_{R}(t)|& =2\Big|\mathbf{Im}\int_{\mathbb{R}^{d}}\overline{u(t)}\nabla u(t)\cdot \nabla \phi _{R}(x+\frac{x(0)}{\mu(0)})\,dx\notag\\
    &-\mathbf{Im}\int_{\R^d}\overline{\tilde{W}(t)}\cdot \nabla \tilde{W}(t)\nabla \phi _{R}(x+\frac{x(0)}{\mu(0)})\,dx\Big|\notag\\
    &\lesssim R\int_{|x+\frac{x(0)}{\mu(0)}|\leq 2R} |\overline{u(t)}\nabla u(t)-\overline{\tilde{W}(t)}\cdot \nabla \tilde{W}(t)|\, dx\notag\\
    & \lesssim R^2 \{\|W\|_{\dot H^1}+\| u\|_{L_t^\infty \dot H_x^1}\}\{\|u-\tilde W\|_{L_t^\infty \dot H_x^1} + \|u-\tilde W\|_{L_t^\infty L_x^{2^*}}\}\notag \\
    & \lesssim R^2 \|u-\tilde W\|_{L_t^\infty \dot H_x^1}\lesssim R^2 \bd(u(t)), 
\end{align}
where the implicit constant depends only on $d$ and $a$. 

Recalling that $R=\frac{C_1}{\mu(0)}$, the fundamental theorem of calculus gives
\[
\int_0^{T}\bd(u(t))dt\leq \frac{d-2}{8}\int_0^{T}\partial_{tt}V_{R}dt\leq \frac{d-2}{4}\max_{t\in[0,T]}|\partial_{t}V(t)|\lesssim_{d,a} \frac{C_1^2}{\mu(0)^2}\delta_0. 
\]
Combining this with \eqref{mod2} and the bootstrap assumption \eqref{bootstrap1}, we deduce
\begin{align*}
    \left|\ln \mu(t)-\ln \mu(0)\right|\leq \int_{0}^T\Big|\frac{\mu^{'}(t)}{\mu(t)}\Big|dt&\lesssim \max_{t\in[0,T]}\mu(t)^2\int_{0}^T\bd(u(t)) dt\lesssim_{d,a} 4C_1^2\delta_0,
\end{align*}
and
\begin{align*}
    \left|\frac{x(t)}{\mu(t)}-\frac{x(0)}{\mu(0)}\right|\leq \int_{0}^T\Big|\frac{x^{'}(t)}{\mu(t)}-\frac{\mu^{'}(t)}{\mu(t)^2}x(t)\Big|\,dt&\lesssim \max_{t\in[0,T]}{\mu(t)}\int_{0}^T\bd(u(t)) dt\lesssim_{d,a} 2C_1^2\delta_0\frac{1}{\mu(0)}.
\end{align*}
Thus, for $\delta_0$ is sufficiently small, the improved estimate \eqref{boostrap2} follows. 

It remains to prove \eqref{bst1} and \eqref{bst2}. We begin with \eqref{bst1}. By a change of variables, we obtain
\begin{align*}
\biggl|\int_{|x|\geq R} & \Big|\nabla u\big(t,x-\frac{x(0)}{\mu(0)}\big)\Big|^2 - \Big|\nabla \tilde W\big(t,x-\frac{x(0)}{\mu(0)}\big)\Big|^2 \,dx\biggr| \\
& \lesssim\bigl\{ \|\nabla u(t)\|_{L^2(|x+\frac{x(0)}{\mu(0)}|\geq R)} + \|\nabla W\|_{L^2(|\frac{x-x(t)}{\mu(t)}+\frac{x(0)}{\mu(0)}|\geq R)} \bigr\} \|\nabla v(t)\|_{L^2}.
\end{align*}
Using $R=\frac{C_1}{\mu(0)}$ and the bootstrap assumption \eqref{bootstrap1}, we have
\[
\Big\{\big|\frac{x-x(t)}{\mu(t)}+\frac{x(0)}{\mu(0)}\big|\geq R\Big\} \subset \Big\{\big|\frac{x}{\mu(t)}\big|\geq R-\frac{1}{\mu(0)}\Big\}\subset \Big\{|x|\geq\frac{C_1-1}{2}\Big\}.  
\]
Hence, for any $\eta>0$, there exists $C(\eta)$ such that for any $C_1\geq C(\eta)$
\[
\|\nabla W\|_{L^2(|\frac{x-x(t)}{\mu(t)}+\frac{x(0)}{\mu(0)}|\geq R}\leq \eta. 
\]
Moreover, by \eqref{mod1} and changing variables,
\[
\|\nabla u(t)\|_{{L^2(|x+\frac{x(0)}{\mu(0)}|\geq R)}}\leq \|\nabla W\|_{L^2(|\frac{x-x(t)}{\mu(t)}+\frac{x(0)}{\mu(0)}|\geq R)}+\|\nabla v(t)\|_{L_x^2} \lesssim \eta+\bd(u(t)).
\]
Therefore
\[
\biggl|\int_{|x|\geq R}  \Big|\nabla u\big(t,x-\frac{x(0)}{\mu(0)}\big)\Big|^2 - \Big|\nabla \tilde W\big(t,x-\frac{x(0)}{\mu(0)}\big)\Big|^2 \,dx\biggr| \lesssim (\eta+\delta_0) \bd(u(t)). 
\]
The contribution of the potential energy in \eqref{V2} is handled similarly. By H\"older's inequality and Sobolev embedding, 
\begin{align*}
 &\biggl| \int_{|x|\geq R}  \bigl |u(x-\frac{x(0)}{\mu(0)})|^{2^*} \,dx - \int_{|x|\geq R} |\tilde W(x-\frac{x(0)}{\mu(0)})\bigr|^{2^*}\,dx \biggr| \\
 \lesssim &\bigl\{ \|u\|_{L_x^{2^*}(|x+\frac{x(0)}{\mu(0)}|\geq R)}^{\frac{d+2}{d-2}} + \|W\|_{L_x^{2^*}(|\frac{x-x(t)}{\mu(t)}+\frac{x(0)}{\mu(0)}|\geq R)}^{\frac{d+2}{d-2}}\} \|v(t)\|_{L_x^{2^*}} \\
 \lesssim &(\eta^{\frac{d+2}{d-2}}+\delta_0^{\frac{d+2}{d-2}})\|v(t)\|_{\dot H^1} \\
  \lesssim & (\eta^{\frac{d+2}{d-2}}+\delta_0^{\frac{d+2}{d-2}})\bd(u(t)). 
\end{align*}
The remaining terms in \eqref{V2} are treated analogously, yielding
\[
\left|A_R^0\Big(u\big(t,x-\frac{x(0)}{\mu(0)}\big)\Big)-A_R^0\Big(\tilde{W}\big(x-\frac{x(0)}{\mu(0)}\big)\Big)\right| \lesssim [\eta+\delta_0]\bd(u(t)).
\]
Choosing $\eta>0$, $\delta_0>0$ sufficiently small proves \eqref{bst1}.

We now prove \eqref{bst2}. Since $\phi_{R}$ is supported in $\{|x|\leq 2R\}=\{|x|\leq \frac{2C_1} {\mu(0)}\}$, if  $|x(0)|\geq C_2\geq 8C_1$, then for all $x$ in the support of $ \phi_{R}$,
\[
|\nabla\phi_{R}|\leq 2 |x|\leq 2 \bigl |x-\tfrac{x(0)}{\mu(0)}\bigr|.
\]
Consequently, 
\[
\biggl |4a\int_{\mathbb{R}^{d}}[\tfrac{x-\frac{x(0)}{\mu(0)}}{|x-\frac{x(0)}{\mu(0)}|^{4}}\cdot \nabla \phi
_{R}-\tfrac{2|x-\frac{x(0)}{\mu(0)}|^2}{|x-\frac{x(0)}{\mu(0)}|^4}]|u(t,x-\frac{x(0)}{\mu(0)})|^{2}\,dx\biggr|\lesssim \int_{\R^d}\frac{|u|^2}{|x|^2} dx\lesssim \bd(u(t))^2\lesssim \delta_0\bd(u(t)).
\]
Taking $\delta_0>0$ sufficiently small establishes \eqref{bst2} and hence completes the proof of Proposition \ref{bootstrap}.
\end{proof}

With Proposition \eqref{bootstrap} in hand, we are ready to prove Theorem \ref{rigidity}.
\begin{proof}[Proof of Theorem \ref{rigidity}]
   We first show that for $\delta_0$ sufficiently small, one must have
   \begin{equation}\label{final1}
       |x(0)|\geq C_2.
   \end{equation}
   Indeed, by \eqref{decompz0}, \eqref{mod1} and Hardy's inequality, we obtain
   \begin{align}\label{mnz}
       \int_{\R^d} \frac{|u(t)|^2}{|x|^2} dx=\int_{\R^d} \frac{\biggl |[W+v]_{[-\theta(t),\frac{1}{\mu(t)}, -\frac{x(t)}{\mu(t)}]}\biggr|^2}{|x|^2}\lesssim \bd(u(t))^2
   \end{align} 
   and
   \begin{align}\label{mnz2}
       \int_{\R^d} \frac{|v_{[-\theta(t),\frac{1}{\mu(t)}, -\frac{x(t)}{\mu(t)}]}|^2}{|x|^2}\,dx\lesssim \|v\|_{\dot{H}^1}^2\lesssim \bd(u(t))^2.
   \end{align}
   Combining \eqref{mnz} and \eqref{mnz2}, we deduce that for $\delta_0\ll1$,
   \[
   \int_{\R^d}\frac{|W(x+x(t))|^2}{|x|^2}\,dx=\int_{\R^d} \frac{|W_{[-\theta(t),\frac{1}{\mu(t)}, -\frac{x(t)}{\mu(t)}]}|^2}{|x|^2}\lesssim  \bd(u(t)). 
   \]
This together with the definition of $W$ in \eqref{def:Wa} implies 
      \begin{equation}\label{final2}
      |x(t)|\gtrsim (\bd(u(t)))^{-\frac{1}{d-2}},
   \end{equation}
   and hence \eqref{final1} follows by taking $\delta_0$ sufficiently small. 
   
  On the one hand, by Proposition \ref{bootstrap}, for some $M>0$, we have
   \begin{align}\label{559}
   \mu(t)\sim \mu(0)\qtq{and} |x(t)|\leq M,\quad  \forall t\ge 0.
   \end{align}

  On the other hand, by Lemma \ref{sequencec}, there exists a sequence $t_n\nearrow\infty$ such that $\bd(u(t_n))\to 0$. This together with \eqref{final2} yields
    \begin{equation*}
      |x(t_n)|\to \infty,\quad \textit{as $t_n\to \infty$},
    \end{equation*}
    which clearly contradicts the uniform bound in \eqref{559}.   
\end{proof}

\section{Proof of the main result}\label{S:reduction}

In this section, we complete the proof of the main result (Theorem~\ref{thm1}).
\begin{proof}[Proof of Theorem~\ref{thm1}] 

Arguing as in Section \ref{S:compact}, together with the time-reversibility, the failure of Theorem \ref{thm1} yields the existence of a non-scattering, sub-threshold solution satisfying  
\begin{align*}
E_a(u)=E_0(W), \ \|u\|_{\dot{H}_a^1}<\|W\|_{\dot{H}^1}, \|u\|_{S([0, \infty))}=\infty,
\end{align*}
and $u$ is almost periodic modulo symmetries on $[0,\infty)$. Moreover, by Lemma \ref{sequencec}, there exists a sequence $t_n\nearrow\infty$ such that 
\begin{equation}\label{1048}
\bd (u(t_n))\to 0, \ \mbox{ as } n\to \infty. 
\end{equation}
We distinguish two cases according to the behavior of $\bd(u(t))$. 

\medskip
\noindent \textbf{Case 1. } There exists $t_0>0$ such that 
\begin{equation}\label{1044}
\sup_{t\in[t_0,\infty)}\bd(u(t)) \leq \delta_0.
\end{equation}

\medskip
\noindent\textbf{Case 2}. There exists a sequence $t_n^-\nearrow \infty$ such that  
\begin{equation}\label{1049}
\bd(u(t_n^-))>\delta_0\qtq{for all $n$.}
\end{equation}

\medskip
Case 1 is ruled out immediately applying the rigidity result Theorem \ref{rigidity} together with the time translation symmetry of \eqref{NLSa}. 

It remains to exclude Case 2. Passing to a subsequence if necessary, we may assume
\[
\bd(u(t_n)) < \delta_0, \qtq{and}t_n^- < t_n < t_{n+1}^- \qtq{for all $n$.}
\]
By continuity, there exits $t_n^+ \in (t_n^-, t_n)$ such that for the first time after $t_n^-$, 
\[
\bd(u(t_n^+))=\delta_0. 
\]
This together with the contrapositive statement of Lemma \ref{distant} imply the following: After passing to a futher subsequence, there exist $\lambda_n>0$ and $w_0\in\dot{H}^1(\R^d)$ such that
\begin{align}\label{1122}
\|\lambda_n^{\frac{d-2}{2}}u(t_n^+, \lambda_n x)- w_0\|_{\dot{H}_a^1}\to 0,\quad \text{ as }n\to\infty. 
\end{align}

Let $w$ be the solution to \eqref{NLSa} with $w(0)=w_0$. Then \eqref{1122} implies that $w$ is also a sub-threshold solution satisfying
\begin{align}\label{220}
E_a(w)=E_0(W) \qtq{and}\|w\|_{\dot{H}_a^1}<\|W\|_{\dot{H}^1},
\end{align}
hence is global by Proposition~\ref{blowup}.  

We claim that
\begin{align}
\|w\|_{S([0,\infty))}=\infty \label{217}, \mbox{ and}\\
 \sup_{t\in[0,\infty)} \bd(w(t))\leq \delta_0.\label{218}
\end{align}
Assuming \eqref{217} \eqref{218} for the moment, we obtain a contradiction with Theorem \ref{rigidity}, hence exclude Case 2 and complete the proof our main theorem. 

To prove \eqref{217}, define
\[
\tilde u_n(t,x) = \lambda_n^{\frac{d-2}{2}} u(t_n^+ + \lambda_n^2 t,\lambda_n x).
\]
Then $\tilde u_n$ solves \eqref{NLSa} and satisfies
\begin{align}\label{1143}
\|\tilde u_n(0)-w_0\|_{\dot H^1}\to 0, \ \|\tilde u_n\|_{S([0, \infty))}=\|u\|_{S([t_n^+,\infty))}=\infty. 
\end{align}
Thus \eqref{217} follows from the stability result Lemma~\ref{stability}. 

Next, we prove \eqref{218}. We first claim that for any fixed $T>0$, there exists $n_0(T)$ such that
\begin{equation}\label{239}
\lambda_n^{-2}(t_n-t_n^+)\geq T, \ \forall \ n\ge n_0(T).
\end{equation}
 Suppose this fails. Then, there exists a subsequence in $n$ such that
\[
\lambda_n^{-2}(t_n-t_n^+)< T. 
\]
By the stability result (Lemma \ref{stability}) and the first convergence in \eqref{1143}, we obtain
\begin{equation}\label{256}
\|\nabla[\tilde{u}_n-w]\|_{L_t^{\infty}L_x^2([0,T]\times\R^d)}\to 0 \mbox{ as } n\to\infty.
\end{equation}

Using this norm convergence and a change of variables we deduce 
\begin{align*}
\inf_{t\in [t_n^+,t_n]}\bd(u(t))=\inf_{t\in[0,\lambda_n^{-2}(t_n-t_n^+)]}\bd(\tilde{u}_n(t))>\frac{1}{2}\inf_{t\in[0,T]}\bd(w(t))>0,
\end{align*}
which contradicts $\bd(u(t_n))\to0$. Hence \eqref{239} is proved.

Finally, using \eqref{239}, the triangle inequality and \eqref{256}, we obtain for the fixed $T>0$,
\begin{align*}
\bd(w(T)) & \leq \bd(\tilde u_n(T)) +  \bigl |\|\tilde u_n(T)\|_{\dot{H}_a^1}^2 - \|w(T)\|_{\dot{H}_a^1}^2 \bigr|\\
& \leq \bd(u(t_n^++\lambda_n^2 T)) + C\|{\tilde u_n(T)}-w(T)\|_{\dot{H}_a^1} \\
& \leq \sup_{t\in [t_n^+,t_n]}\bd(u(t)) + o_n(1)\\
& \leq \delta_0 + o_n(1).
\end{align*}
 Letting $n\to\infty$, we conclude $\bd(w(T))\leq \delta_0$. Since $T>0$ is arbitrary, this proves \eqref{218} and completes the proof of Theorem \ref{thm1}. 
 \end{proof}


\end{document}